\documentclass[10pt]{amsart}
\usepackage{amssymb}

\usepackage{blindtext,stmaryrd}
\usepackage[utf8]{inputenc}
\usepackage{xcolor}
\usepackage{hyperref}
\usepackage{enumerate}
\newtheorem{thm}{Theorem}[section]
\theoremstyle{definition}
\newtheorem{ques}{Question}

\newtheorem{cor}[thm]{Corollary}
\newtheorem{lem}[thm]{Lemma}
\newtheorem{prop}[thm]{Proposition}
\theoremstyle{definition}
\newtheorem{defn}[thm]{Definition}
\theoremstyle{definition}
\newtheorem{example}[thm]{Example}
\theoremstyle{remark}
\newtheorem{rem}[thm]{Remark}
\numberwithin{equation}{section}

\newcommand{\mcK}{\mathcal{K}}
\DeclareMathOperator{\Aut}{Aut}
\DeclareMathOperator{\Iso}{Iso}

\DeclareMathOperator{\rng}{rng}
\DeclareMathOperator{\dom}{dom}
\DeclareMathOperator{\Sym}{Sym}
\DeclareMathOperator{\Age}{Age}
\DeclareMathOperator{\diam}{diam}

\newcommand{\Q}{{\mathbb Q}}
\newcommand{\U}{{\mathbb U}}
\newcommand{\Z}{{\mathbb Z}}

\newcommand{\actson}{\curvearrowright}
\newcommand{\N}{{\mathbb N}}
\newcommand{\sL}{\mathcal{L}}

\begin{document}

\title[Automorphism Groups of Ultraextensive Spaces]{Dense locally finite 
subgroups of Automorphism Groups of Ultraextensive Spaces} 

\author[M. Etedadialiabadi]{Mahmood Etedadialiabadi}

\address{Department of Mathematics, University of North Texas, 1155 Union Circle \#311430, Denton, TX 76203, USA}

\email{mahmood.etedadialiabadi@unt.edu}

\author[S. Gao]{Su Gao}

\address{Department of Mathematics, University of North Texas, 1155 Union Circle \#311430, Denton, TX 76203, USA}\email{sgao@unt.edu}

\author[F. Le Maître]{François Le Maître}
\address{Université Paris Diderot, Sorbonne Université, CNRS, Institut de 
Mathématiques de Jussieu-Paris Rive Gauche, IMJ-PRG, F-75013, Paris, France.}
\email{f.lemaitre@math.univ-paris-diderot.fr}

\author[J. Melleray]{Julien Melleray}
\address{Université de Lyon, Université Claude Bernard - Lyon 1, Institut 
Camille Jordan, CNRS UMR 5208, 43 Boulevard du 11 novembre 1928, 69622 
Villeurbanne Cedex, France}
  \email{melleray@math.univ-lyon1.fr}

\thanks{The second author's research was partially supported by the NSF grant DMS-1800323.}



\begin{abstract}
We verify a conjecture of Vershik by showing that Hall's universal countable 
locally finite group can be embedded as a dense subgroup in the isometry group 
of the Urysohn space  and in the automorphism group of the random graph. In 
fact, we show the same for all automorphism groups of known infinite 
ultraextensive spaces. These include, in addition, the isometry group of the 
rational Urysohn space, the isometry group of the ultrametric Urysohn spaces, 
and the automorphism group of the universal $K_n$-free graph for all $n\geq 3$. 
Furthermore, we show that finite group actions on finite metric spaces or 
finite relational structures form a Fraïssé class, where Hall's group appears 
as the acting group of the Fraïssé limit. We also embed continuum many 
non-isomorphic countable universal locally finite groups into the isometry 
groups of various Urysohn spaces, and show that all dense countable subgroups 
of these groups are mixed identity free (MIF). Finally, we give a 
characterization of the isomorphism type of the isometry group of the Urysohn 
$\Delta$-metric spaces in terms of the distance value set $\Delta$.
\end{abstract}

\maketitle


\section{Introduction}
The concepts of ultraextensive metric spaces and ultraextensive relational 
structures were introduced in \cite{EG} and \cite{EG2}, respectively, to 
capture some common properties possessed by many Fraïssé limits. 
One of the main properties of an ultraextensive space is 
that its automorphism group contains a countable dense locally finite subgroup. 
Known examples of ultraextensive spaces include the universal Urysohn space 
$\mathbb{U}$, the universal rational Urysohn space $\mathbb{QU}$, the universal 
ultrametric Urysohn spaces, the random graph $\mathcal{R}$, and the universal 
$K_n$-free graphs $H_n$ (also known as Henson graphs). That the automorphism 
groups (or isometry groups) of these spaces contain a countable dense locally 
finite subgroup was proved in Bhattacharjee--Macpherson \cite{BM}, Pestov 
\cite{P}, Rosendal \cite{R}, and Siniora--Solecki \cite{S2}. In this paper we 
consider a concept of universal $\Delta$-metric spaces that unifies the study 
of the Urysohn space $\mathbb{U}$, the rational Urysohn space $\mathbb{QU}$ and 
the random graph $\mathcal{R}$. Given a distance value set $\Delta$, the 
universal $\Delta$-metric space $\mathbb{U}_\Delta$ is an ultraextensive metric 
space. 

Vershik \cite{V} conjectured that $\Iso(\mathbb{U})$, the isometry group of the 
Urysohn space, contains Hall's universal countable locally finite group 
$\mathbb{H}$ as a dense subgroup. He made the same conjecture for the 
automorphism group of the countable random graph. Our first main result of this 
paper is to confirm Vershik's conjecture. In fact, we establish this for all 
known examples of infinite ultraextensive spaces.

\begin{thm}\label{mainthm}
The following groups contain Hall's universal countable locally finite group $\mathbb{H}$ as a dense subgroup:
\begin{enumerate}
\item $\Iso(\mathbb{U})$, the isometry group of the Urysohn space;
\item $\Iso(\mathbb{QU})$, the isometry group of the rational Urysohn space;
\item $\Iso(\mathbb{U}_\Delta)$, the isometry group of the universal $\Delta$-metric space, for any distance value set $\Delta$;
\item Isometry groups of ultrametric Urysohn spaces;
\item $\Aut(\mathcal{R})$, the automorphism group of the random graph; and
\item $\Aut(H_n)$, the automorphism group of the universal $K_n$-free graph, for any $n\geq 3$.
\end{enumerate}
\end{thm}

In fact, we show that $\mathbb{H}$ appears canonically as a dense subgroup in 
these automorphism groups via the following theorems.

\begin{thm}\label{FraisseTheorem1} Let $\mcK_\Delta$ be the class of all 
structures $(X, G)$ such that $X$ is a finite $\Delta$-metric space, $G$ is a 
finite group, and $G$ acts on $X$ by isometries. Then $\mcK_\Delta$ is a 
Fraïssé class. Letting $(X_\Delta, H_\Delta)$ be the Fraïssé limit of 
$\mcK_\Delta$, then $X_\Delta$ is isometric to $\U_\Delta$, $H_\Delta$ is 
isomorphic to $\mathbb{H}$, and $H_\Delta$ is dense in $\Iso(X_\Delta)$.
\end{thm}

We note here that this result is related to some recent work of Doucha \cite{D}, see the comments at the end of Section \ref{VC}.

\begin{thm}\label{FraisseTheorem2} Let $\sL$ be a finite relational language. 
Let $\mathcal{T}$ be a finite set of finite $\sL$-structures each of which is a 
Gaifman clique. Let $\mcK$ be the class of all pairs $(M, G)$ such that $M$ is 
a finite $\mathcal{T}$-free $\sL$-structure, $G$ is a finite group, and $G$ 
acts on $X$ by isomorphisms. Then $\mcK$ is a Fraïssé class. Letting 
$(N_\infty, H_\infty)$ be the Fraïssé limit of $\mcK$, then $N_\infty$ is 
isomorphic to the universal $\mathcal{T}$-free $\sL$-structure, $H_\infty$ is 
isomorphic to $\mathbb{H}$, and $H_\infty$ is dense in $\Aut(N_\infty)$.
\end{thm}

In the proof of Theorem~\ref{FraisseTheorem1} we use a result of Rosendal 
\cite{R} that characterizes the RZ-property (after Ribes--Zalesski\u{\i}) by an 
extension property for finite metric spaces. Then, for 
Theorem~\ref{FraisseTheorem2} we need to use a concept of HL-property (after 
Herwig--Lascar) and a characterization of the HL-property in the spirit of 
Rosendal's result, both developed in \cite{EG2}. For both the RZ-property and 
the HL-property, we do need to establish some new results about their closure 
under finite-index extensions.

Next we turn to the problem of constructing many non-isomorphic dense locally 
finite subgroups of $\Iso(\U_\Delta)$, the isometry group of the Urysohn 
$\Delta$-metric spaces. We define a notion of omnigenous groups, which can be 
viewed as a natural generalization of Hall's group. We then show that there are 
many such groups, and that they are all densely embeddable.

\begin{thm}\label{omnigenousconstruction}
	There are continuum many non-isomorphic countable omnigenous groups each of 
	which is universal for countable locally finite groups.
\end{thm}

\begin{thm}\label{omnigenousmainthm}
Every countable omnigenous group is embeddable into $\Iso(\U_\Delta)$ as a 
dense subgroup.
\end{thm}

In an effort to characterize all countable dense subgroups of 
$\Iso(\U_\Delta)$, we consider some notions from the point of view of model 
theory and combinatorial group theory. In particular, we consider the property 
of being ``mixed-identity free'' (MIF) recently studied by Hull--Osin \cite{HO} 
and prove 
that any countable dense subgroup of $\Iso(\U_\Delta)$, when $|\Delta|\geq 2$, 
must be MIF. 

\begin{thm}
For any countable distance value set $\Delta$ with $|\Delta|\geq 2$, 
$\Iso(\mathbb{U}_\Delta)$ as well as any of its dense subgroup must be MIF.
\end{thm}

This theorem is false when $|\Delta|=1$. In that case $\Iso(\U_\Delta)$ is just 
the permutation group 
$S_\infty$ and it is known that $S_\infty$ is not MIF since it contains the non 
MIF group of finitely supported permutations as a dense subgroup (cf. Theorem 5.9 of
Hull--Osin 
\cite{HO}). Note that our result also provides an elementary proof of 
the fact 
that the group of finitely supported permutations
cannot arise as a dense subgroup of $\Iso(\mathbb{U}_\Delta)$ as soon as 
$|\Delta|\geq 2$ (cf. Theorem 8.1 of \cite{FLMMM}).
\begin{rem}
 The theorem also yields a continuum family of universal countable locally 
 finite 
groups that are not embeddable as dense subgroups of $\Iso(\U_\Delta)$, namely 
groups of the form $\mathbb{H}\oplus A$, where $A$ is a 
nontrivial abelian $p$-group. This can also be seen using the fact that 
$\Iso(\U_\Delta)$ is always a topologically simple non abelian group (see Thm. 
\ref{t:topsimple}), and no such 
group can contain a dense subgroup which decomposes as a nontrivial direct 
product.
\end{rem}

Furthermore, we introduce a new notion for locally finite groups which we  call 
$\infty$-MIF, and show that it is actually
equivalent to being omnigenous.\\ 

It would be very interesting to be able to distinguish topological groups of 
the form $\Iso(\U_\Delta)$ by looking at their list of countable dense 
subgroups. As a first step towards this, it is natural to ask which 
$\Iso(\U_\Delta)$ can be densely embedded into another $\Iso(\U_{\Lambda})$. 
Indeed if 
so then $\Iso(\U_{\Lambda})$ will contain at least as many countable dense 
subgroups as $\Iso(\U_\Delta)$. Our next result shows that these dense 
embeddings only occur in the obvious case, namely when $\Iso(\U_\Delta)$ and 
$\Iso(\U_{\Lambda})$ are isomorphic, and provides a natural characterization in 
terms of the distance sets $\Delta$ and $\Lambda$ for this to happen.
 This uses the following notion: we call $(d_1, d_2, d_3)\in\Delta^3$ a 
 \emph{$\Delta$-triangle} if there is a 
 metric space $(X, d)$ with $X=\{x, y, z\}$ such that $d_1=d(x, y)$, 
 $d_2=d(y,z)$ and $d_3=d(z, x)$.

\begin{thm}\label{isoIso} Let $\Delta$ and $\Lambda$ be countable distance value sets. Then the following are equivalent:
\begin{enumerate}[(i)]
\item There is a continuous homomorphism $\Iso(\U_\Delta)\to \Iso(\U_\Lambda)$ 
with dense range; 
\item $\Iso(U_\Delta)$ and $\Iso(\U_\Lambda)$ are isomorphic as abstract groups;
\item $\Iso(U_\Delta)$ and $\Iso(\U_\Lambda)$ are isomorphic as topological 
groups;
\item There exists a bijection $\theta: \Delta\to \Lambda$ such that
for any triple $(d_1, d_2, d_3)\in \Delta^3$, $(d_1, d_2, d_3)$ is a $\Delta$-triangle iff $(\theta(d_1), \theta(d_2), \theta(d_3))$ is a $\Lambda$-triangle.
\end{enumerate}
\end{thm}

The rest of the paper is organized as follows. In Section~\ref{sec: ultraextensive} we cover some preliminaries and verify that $\U_\Delta$ is ultraextensive for any countable distance value set $\Delta$. In Section~\ref{VC} we prove Theorem~\ref{FraisseTheorem1}. In Section~\ref{sec:structures} we develop results about the HL-property of groups and prove Theorem~\ref{FraisseTheorem2}. In Section~\ref{sec:many} we study the notion of omnigenous groups and prove Theorems~\ref{omnigenousmainthm} and \ref{omnigenousconstruction}. We apply these results also to the isometry groups of ultrametric Urysohn spaces. In Section~\ref{sec:all} we study the notions of discerning types, discerning structures, MIF groups, and $\infty$-MIF groups. In Section~\ref{sec:cha} we prove that $\Iso(\U_\Delta)$, as well as pointwise stabilizers on $\U_\Delta$, are topologically simple; this is used to establish Theorem~\ref{isoIso}. Finally, in Section~\ref{sec:open} we pose some open problems.

\section{Ultraextensive Metric Spaces\label{sec: ultraextensive}}

\subsection{Basics of Fraïssé theory}
We briefly recall the basic concepts of Fraïssé theory. Throughout this paper 
let $\mathcal{L}$ be a countable language.

\begin{defn} Let $M$ be a countable $\mathcal{L}$-structure. A \emph{partial automorphism} of $M$ is an isomorphism $g \colon A \to B$, where $A$ and $B$ are finitely generated substructures of $M$.

The structure $M$ is said to be \emph{ultrahomogeneous} if every partial automorphism of $M$ extends to an automorphism of $M$.
\end{defn}

In the cases considered in this paper, finitely generated substructures are always finite. For example, this happens when $\mathcal{L}$ is a relational language. Another case we will consider is when $\mathcal{L}$ is the (finite) language of group theory, and $M$ is a countable locally finite group. We will assume this property tacitly in all of our discussions.

\begin{defn}
Let $M$ be a countable $\mathcal{L}$-structure. The \emph{age} of $M$, denoted $\Age(M)$, is the class of all finite substructures of $M$ (considered up to isomorphism).
\end{defn}

The age of any countable $\mathcal{L}$-structure contains only countably many members up to isomorphism; also, any two members of the age embed in a third one (the \emph{joint embedding property}) and whenever $A \in \text{Age}(M)$ and $B$ is a substructure of $A$ then also $B \in \Age(M)$ (the \emph{hereditary property}). Ages of ultrahomogeneous structures are characterized by an additional condition.

\begin{defn}
Let $\mcK$ be a class of $\mathcal{L}$-structures. We say that $\mcK$ has the \emph{amalgamation property} if, for any $A,B,C \in \mcK$ and any embedding $\beta \colon A \to B$, $\gamma \colon A \to C$, there exists $D \in \mcK$ and embeddings $\beta' \colon B \to D$ and $\gamma' \colon C \to D$ such that $\beta'\circ \beta(a)= \gamma' \circ \gamma(a)$ for all $a \in A$.

$\mcK$ has the \emph{strong amalgamation property} if in the above definition we have in addition $\beta'(B)\cap \gamma'(C)=\beta'\circ\beta(A)$.
\end{defn}

\begin{thm}[Fraïssé]
The age of any ultrahomogeneous $\mathcal{L}$-structure satisfies the amalgamation property. Conversely, if $\mcK$ is a countable (up to isomorphism) class of finite $\mathcal{L}$-structures which has the joint embedding, hereditary and amalgamation properties then there exists a unique (up to isomorphism), ultrahomogeneous countable $\mathcal{L}$-structure $M$ such that $\Age(M)=\mcK$. 
\end{thm}

A class $\mcK$ satisfying the assumptions of the theorem is called a 
\emph{Fraïssé class}, and the unique structure $M$ above is called the 
\emph{Fraïssé limit} of $\mcK$. It is also characterized by the statement that $\Age(M) = \mcK$, and for any $A \subseteq M$, any $B \in \mcK$ and any embedding $\varphi \colon A \to B$ there exists an embedding $\psi \colon B \to M$ such that $\psi(\varphi(a))=a$ for all $a \in A$.

\subsection{$\Delta$-metric spaces\label{subsectionDelta}}

\begin{defn}
A \emph{distance value set} is a nonempty subset $\Delta$ of the open interval $(0,+\infty)$, such that 
$$\forall x,y \in \Delta \quad \min(x+y,\sup(\Delta)) \in \Delta\ . $$
A \emph{$\Delta$-metric space} is a metric space whose nonzero distances belong to $\Delta$. 
\end{defn}

In particular, when $\Delta$ is bounded the definition implies that $\sup(\Delta) \in \Delta$. The definition above is a particular case of what Conant \cite{C17} calls a \emph{distance monoid}, and our constructions could work for some more general distance monoids. 
For simplicity, we choose to work only in this more restricted case.

In general, for a metric space $(M,d)$, the isometry group $\Iso(M, d)$ is endowed with the pointwise convergence topology, i.e. $g_n\to g$ iff $d(g_n(x),g(x))\to 0$ for all $x \in M$. When $(M, d)$ is a separable complete metric space, $\Iso(M, d)$ becomes a Polish group, and we often write $\Iso(M)$ instead of $\Iso(M, d)$.

In case $(M, d)$ is \emph{countable}, we use the following important convention. We will use $\Iso(M)$ to denote the group $\Iso(M, d)$ but we will view it as a subset of the permutation group $\Sym(M)$ on $M$. As such it will become a closed subgroup of $\Sym(M)$, and $\Iso(M)$ will be endowed with the subspace topology of $\Sym(M)$, which we will refer to as the \emph{permutation group topology}. This again makes $\Iso(M)$ a Polish group. 
A basis of neighborhoods of $1$ for this topology is given by pointwise stabilizers of finite tuples of elements of $M$.

To apply Fraïssé theory, we will assume throughout this paper that $\Delta$ is 
a \emph{countable} distance value set. 

Now any $\Delta$-metric space $(M,d)$ may be viewed as a first-order structure, in a countable relational language with a binary relational symbol $R_s$ for each element $s$ of $\Delta$, by declaring that
$$  M \models R_s(x,y)  \mbox{\ iff\ }  d(x,y)=s. $$ 
Since there should be no risk of confusion, we will be using the distance function instead of those binary relational symbols. 

\begin{lem} For any distance value set $\Delta$, the class of finite $\Delta$-metric spaces has the strong amalgamation property. 
\end{lem}

\begin{proof}
Assume that $A,B,C$ are finite $\Delta$-metric spaces and $A$ is a subspace of both $B$ and $C$. Let $D$ denote the union of $B$ and $C$, where both copies of $A$ are identified and the values of the metric on $B$ and $C$ are imposed (and coincide on $A$). We need to define $d(b,c)$ for $b \in B \setminus A$ and $c \in C \setminus A$. If $A$ is empty, then we let $m$ be the maximum value taken by $d$ on either $B$ or $C$, and set $d(b,c)=m$ for any $b$ and $c$. If $A$ is nonempty, then we set 
\[d(b,c)= \min \{ \min \{d(b,a)+d(a,c) \colon a \in A\}, \sup(\Delta)\}. \]
Then $D$ is a $\Delta$-metric space.
\end{proof}

Thus if $\Delta$ is a countable distance value set, then the class of finite 
$\Delta$-metric spaces is a Fraïssé class. We denote by $\U_\Delta$ the Fraïssé 
limit of this class, which is itself a countable $\Delta$-metric space. We 
emphasize that $G_\Delta=\Iso(\U_\Delta)$ is endowed with the permutation group 
topology.

We will need the following characterization of $\U_\Delta$. The property in the proposition is called the \emph{Urysohn property}.

\begin{prop}\label{UrysohnProperty}
Let $\Delta$ be any countable distance value set. The space $\U_\Delta$ is the unique countable $\Delta$-metric space $X$, up to isometry, that satisfies the following 
property: 
\begin{enumerate}
\item[] Given any finite subset $A$ of $X$ and function $f: A\to \Delta$ 
satisfying
$$ |f(a)-f(b)|\leq d(a, b)\leq f(a)+f(b), \forall a, b\in A, $$
there is an $x\in X$ such that $d(x,a)=f(a)$ for all $a\in A$.
\end{enumerate}
\end{prop}
Functions $f:A\to \Delta$ as above are called \emph{Kat\u{e}tov functions} over 
$A$.
We now mention some well-known examples of spaces having the Urysohn property. 
\begin{example} 
\begin{enumerate} 
\item[(1)] $\Delta$ is a singleton. For instance let $\Delta=\{1\}$. Then $\U_{\{1\}}$ is a countable space with the discrete metric $\delta$, where $\delta(x,y)=1$ iff $x\neq y$, and $G_{\{1\}}$ is isomorphic to $\Sym(\N)$ (also denoted $S_\infty$). Here $\U_{\{1\}}$ can also be viewed as the complete graph $K_\N$, while $\Aut(K_\N)=\Sym(\N)$.
\item[(2)] $\Delta=\{0,1,2\}$. In this case $\U_{\{1,2\}}$ is essentially the random graph $\mathcal{R}$. In fact, if we define in $\mathcal{R}$ the metric $d$ by $d(x,y)=1$ iff there is an edge between $x$ and $y$, then $(\mathcal{R},d)$ is isometric with $\U_{\{1,2\}}$. In this case $G_{\{1,2\}}$ is isomorphic to $\Aut(\mathcal{R})$.
\item[(3)] $\Delta=\Q$. Then $\U_\Q$ is the universal rational Urysohn space $\Q\U$, and $G_\Q=\Iso(\mathbb{QU})$.
\end{enumerate}
\end{example}

\subsection{S-extensions}

We recall the notion of S-extension from \cite{EG}. 

Let $(X, d_X)$ and $(Y,d_Y)$ be metric spaces. When there is no danger of confusion, we simply write $X$ for $(X, d_X)$ and $Y$ for $(Y, d_Y)$. We say that $Y$ is an {\it extension} of $X$ if $(X,d_X)$ is a subspace of $(Y, d_Y)$. Interchangeably, we use the same terminology when $Y$ contains an isometric copy of $X$, i.e. when there is an (obvious) isometric embedding from $X$ into $Y$.

A {\it partial isometry} of $X$ is an isometry between two finite subspaces of $X$. The set of all partial isometries of $X$ is denoted as $\mathcal{P}(X)$. $\mathcal{P}(X)$ is a groupoid with the composition $(p,q)\mapsto p\circ q$, where $p\circ q$ is only defined when $\rng(q)=\dom(p)$, and the inverse $p\mapsto p^{-1}$. 

If $Y$ is an extension of $X$, then every partial isometry of $X$ is also a partial isometry of $Y$. In symbols, we have $\mathcal{P}(X)\subseteq \mathcal{P}(Y)$ if $X\subseteq Y$.

If $p, q\in \mathcal{P}(X)$, we say that $q$ {\it extends} $p$, and write $p\subseteq q$, if 
$$\{ (x, p(x))\,:\, x\in\mbox{dom}(p)\}\subseteq \{ (x, q(x))\,:\, x\in\mbox{dom}(q)\}.$$

We let $1_X$ denote the identity isometry on $X$, i.e., $1_X(x)=x$ for all $x\in X$. Let $\mathcal{P}_X$ denote the set of all $p\in \mathcal{P}(X)$ such that $p\not\subseteq 1_X$. We refer to elements of $\mathcal{P}_X$ as {\it nonidentity partial isometries} of $X$. Note that if $p\in \mathcal{P}_X$ then $p^{-1}\in \mathcal{P}_X$.

\begin{defn} Let $X$ be a metric space and $P\subseteq \mathcal{P}_X$ such that $P=P^{-1}$. An \emph{S-extension} of $X$ with respect to $P$  is a pair $(Y, \phi)$, where $Y\supseteq X$ is an extension of $X$, and $\phi: P\to \mbox{Iso}(Y)$ is such that $\phi(p)$ extends $p$ for all $p\in P$. We also require that $\phi(p^{-1})=\phi(p)^{-1}$ for all $p \in P$. When $P=\mathcal{P}_X$ we call $(Y, \phi)$ an \emph{S-extension} of $X$. 
\end{defn}

The following strong notion of coherence was introduced by Solecki (cf. \cite{R} and \cite{S2}). We use a terminology different from Solecki's since we will have to deal with a weaker notion of coherence in the next subsection.

\begin{defn}[Solecki] Let $X$ be a metric space. An S-extension $(Y,\phi)$ of $X$ is \emph{strongly coherent} if for every triple $(p,q,r)$ of partial isometries of $X$ such that $p\circ q=r$, we have $\phi(p)\circ \phi(q)=\phi(r)$. 
\end{defn}

\begin{thm}[Solecki \cite{S09} \cite{R} \cite{S2}]\label{SoleckiTheorem}
Let $\Delta$ be any distance value set and $X$ be a finite $\Delta$-metric space. Then, $X$ has a finite, strongly coherent S-extension $(Y, \phi)$ where $Y$ is a $\Delta$-metric space.
\end{thm}

The observation that finite, strongly coherent S-extensions can be constructed as $\Delta$-metric spaces was explicit in Solecki's unpublished notes \cite{S09} but follows implicitly from all proofs of existence of finite, strongly coherent S-extensions, e.g. in Siniora--Solecki \cite{S2} or Hubi\v{c}ka--Kone\v{c}n\`{y}--Ne\v{s}et\v{r}il \cite{HKN}.

The following lemma highlights the importance of strongly coherent S-extensions. 
\begin{lem}\label{lem:strongcoherence}
Let $X$ be a metric space and $(Y,\phi)$ be a strongly coherent S-extension of $X$. For every $D\subseteq X$, the map $p\mapsto \phi(p)$ gives a group embedding from $\Iso(D)$ into $\Iso(Y)$. 
\end{lem}

\subsection{Ultraextensive $\Delta$-metric spaces}
We recall more notions from \cite{EG}.

For any metric space $X$ and $P\subseteq \mathcal{P}_X$ such that $P=P^{-1}$, we let $\mathbb{F}(P)$ denote the free group generated by $P$, where for any $p\in P$, the inverse of $p$ in $\mathbb{F}(P)$ is $p^{-1}$. If $(Y, \phi)$ is an S-extension of $X$ with respect to $P$, then $\phi$ can be naturally extended to a homomorphism from $\mathbb{F}(P)$ to $\Iso(Y)$. We still use $\phi$ to denote this group homomorphism, i.e., for any $p_1,\dots, p_n\in P$, 
$$ \phi(p_1\cdots p_n)=\phi(p_1)\circ \cdots \circ \phi(p_n). $$

\begin{defn}\label{minimal}
Let $X$ be a metric space and $P\subseteq \mathcal{P}_X$ such that $P=P^{-1}$. An S-extension $(Y,\phi)$ of $X$ with respect to $P$ is \emph{minimal} if for any $y\in Y$ there is $g\in \mathbb{F}(P)$ and $x\in X$ such that 
$y=\phi(g)(x)$.
\end{defn}

\begin{defn}\label{coherence}
Let $X_1\subseteq X_2$ be metric spaces and $(Y_i,\phi_i)$ be an S-extension of $X_i$ for $i=1,2$. We say that $(Y_1,\phi_1)$ and $(Y_2, \phi_2)$ are \emph{coherent} if 
\begin{enumerate}
\item[(i)] $Y_2$ extends $Y_1$, 
\item[(ii)] $\phi_2(p)$ extends $\phi_1(p)$ for all $p\in\mathcal{P}_{X_1}\subseteq \mathcal{P}_{X_2}$, and
\item[(iii)] letting $K_i=\phi_i(\mathcal{P}_{X_i})\subseteq \mbox{Iso}(Y_i)$ for $i=1,2$, and letting $\kappa:K_1\to K_2$ be the map $\kappa(\phi_1(p))=\phi_2(p)$ for all $p\in\mathcal{P}_{X_1}$, then $\kappa$ extends uniquely to a group embedding from $\langle K_1\rangle$ into $\langle K_2\rangle$.
\end{enumerate}
\end{defn}

This notion of coherence is weaker than Solecki's strong coherence, as witnessed by the following lemma.

\begin{lem}\label{lem:weakcoherence} Let $X_1\subseteq X_2$ be metric spaces, $(Y_1, \phi_1)$ be an S-extension of $X_1$, and $(Y_2, \psi)$ be a strongly coherent S-extension of $X_2 \cup Y_1$. Let $\phi_2: \mathcal{P}_{X_2} \to \Iso(Y_2)$ be defined as
$$ \phi_2(p)=\left\{\begin{array}{ll} \psi(\phi_1(p)), & \mbox{ if $p\in \mathcal{P}_{X_1}$,} \\
\psi(p), & \mbox{ if $p\in \mathcal{P}_{X_2}\setminus\mathcal{P}_{X_1}$.}
\end{array}\right. $$
Then $(Y_1, \phi_1)$ and $(Y_2, \phi_2)$ are coherent.
\end{lem}

\begin{proof} From the definition of $\phi_2$ it is clear that for any $p\in\mathcal{P}_{X_1}$, 
$$ p\subseteq \phi_1(p)\subseteq \psi(\phi_1(p))=\phi_2(p). $$
By Lemma~\ref{lem:strongcoherence}, $\psi$ gives a group embedding from $\Iso(Y_1)$ to $\Iso(Y_2)$. On the other hand, $\psi$ coincides with the map $\phi_1(p)\mapsto \phi_2(p)$ for all $p\in\mathcal{P}_{X_1}$. Thus this map extends uniquely to a group embedding from $\langle \mathcal{P}_{X_1}\rangle\leq \Iso(X_1)$ into $\langle \mathcal{P}_{X_2}\rangle\leq\Iso(Y_2)$.
\end{proof}

\begin{defn}\label{uedef} A metric space $U$ is {\it ultraextensive} if 
\begin{enumerate}
\item[(i)] $U$ is ultrahomogeneous, i.e., there is a $\phi$ such that $(U,\phi)$ is an S-extension of $U$;
\item[(ii)] Every finite $X\subseteq U$ has a finite S-extension $(Y,\phi)$ where $Y\subseteq U$;
\item[(iii)] If $X_1\subseteq X_2\subseteq U$ are finite and $(Y_1, \phi_1)$ is a finite minimal S-extension of $X_1$ with $Y_1\subseteq U$, then there is a finite minimal S-extension $(Y_2,\phi_2)$ of $X_2$ such that $Y_2\subseteq U$ and $(Y_1,\phi_1)$ and $(Y_2,\phi_2)$ are coherent.
\end{enumerate}
\end{defn}

One of the main properties of an ultraextensive metric space $U$ is that $\Iso(U)$ always contains a countable dense locally finite subgroup when $U$ is separable (Theorem 1.4 of \cite{EG}). Moreover, the weaker notion of coherence is sufficient for constructing ultraextensive metric spaces. 

\begin{thm} \label{thm:ultraextensiveDelta} Let $\Delta$ be a countable distance value set. Then $\U_\Delta$ is ultraextensive. In particular, $G_\Delta$ contains a countable dense locally finite subgroup.
\end{thm}

\begin{proof} Recall that $\Age(\U_\Delta)$ is the class of all finite $\Delta$-metric spaces. Since $\U_\Delta$ is ultrahomogeneous, Solecki's Theorem~\ref{SoleckiTheorem} gives (ii) of Definition~\ref{uedef}. Similarly, (iii) of Definition~\ref{uedef} follows from Solecki's Theorem~\ref{SoleckiTheorem} and Lemma~\ref{lem:weakcoherence}.
\end{proof}

We remark that Theorem~\ref{thm:ultraextensiveDelta} can be proved without using Solecki's construction of strongly coherent S-extensions. For instance, condition (ii) of Definition~\ref{uedef} for $\Delta$-metric spaces follows implicitly from all proofs of the existence of S-extensions, including Solecki's original proof in \cite{S}. Condition (iii) of Definition~\ref{uedef} for $\Delta$-metric spaces follows implicitly from the proof of Theorem 4.5 of \cite{EG}.

\section{Hall's Group and Vershik's Conjecture\label{VC}}

\subsection{Hall's universal countable locally finite group\label{HallGroup}}
Recall that a locally finite group is a group in which every finitely generated subgroup is finite. The following theorem is due to P. Hall \cite{H}.

\begin{thm}[Hall \cite{H}]\label{HallTheorem}
There exists a countable locally finite group $\mathbb{H}$ that is determined up to isomorphism by the following properties:
\begin{enumerate}
\item[({\sc A})] any finite group can be embedded in $\mathbb{H}$, and 
\item[({\sc B})] any two isomorphic finite subgroups of $\mathbb{H}$ are conjugate by an element of $\mathbb{H}$.
\end{enumerate}
\end{thm} 

It follows easily from the characterizing properties of $\mathbb{H}$ that every countable locally finite group is a subgroup of $\mathbb{H}$. Thus $\mathbb{H}$ is called \emph{Hall's universal countable locally finite group}. For simplicity, we refer to it as \emph{Hall's group}. 

Hall \cite{H} also established the following strengthening of condition ({\sc B}) above.

\begin{prop}[Hall \cite{H}]\label{cong} For every triple $(G_1,G_2,\Psi)$, where $G_1,G_2$ are finite subgroups of $\mathbb{H}$ and $\Psi:G_1\rightarrow G_2$ is a group isomorphism, there exists $h\in \mathbb{H}$ such that for every $g\in G_1$ we have $\Psi(g)=hgh^{-1}$.
\end{prop}

In particular, we see that $\mathbb{H}$ is ultrahomogeneous and universal for 
finite groups: as such, it the Fraïssé limit of the class of finite groups. 
Thus it can also be characterized as follows.

\begin{prop}\label{propertyE}
Let $H$ be a countable locally finite group with the following property: 
\begin{enumerate}
\item[({\sc E})] for every triple $(G_1,G_2,\Psi_1)$, where $G_1\leq G_2$ are finite groups and $\Psi_1:G_1\rightarrow H$ is a group embedding, there exists a group embedding $\Psi_2:G_2\rightarrow H$ such that $\Psi_2\upharpoonright G_1=\Psi_1$.
\end{enumerate}
Then $H$ is isomorphic to $\mathbb{H}$.
\end{prop}



Hall \cite{H} also proved that the commutator group of $\mathbb{H}$ is 
$\mathbb{H}$, and therefore $\mathbb{H}$ has a trivial abelianization. Consider 
the collection of all groups of the form $\mathbb{H}\oplus A$, where $A$ is an 
abelian $p$-group. The abelianization of $\mathbb{H}\oplus A$ is isomorphic to 
$A$. This implies that $\mathbb{H}\oplus A\cong \mathbb{H}\oplus A'$ iff 
$A\cong A'$. Thus there are continuum many non-isomorphic countable locally 
finite groups which are universal for all countable locally finite groups.

\subsection{A proof of Vershik's conjecture}
Vershik's conjecture \cite{V} states that the isometry group of the universal Urysohn space and the automorphism group of the countable random graph each contain a dense subgroup that is isomorphic to $\mathbb{H}$. We will show that Hall's group $\mathbb{H}$ actually arises in some sense as a canonical dense subgroup of $\Iso(\U_\Delta)$.




We will be using the following lemma due to Rosendal (Lemma 16 of \cite{R2}). 

\begin{lem}[Rosendal \cite{R2} ]\label{l:extension}
Let $\Delta$ be a countable distance value set. Let $\Gamma$ be a group, $\Lambda\le \Gamma$ a subgroup. Assume that $X \subseteq Y$ are $\Delta$-metric spaces, and that $\Lambda \actson Y$, $\Gamma \actson X$ are compatible isometric actions. Then there exists a $\Delta$-metric space $Z$ containing $Y$, and an isometric action $\Gamma \actson Z$ compatible with the $\Lambda$-action on $Y$.

Moreover, if $\Gamma$ and $Y$ are both finite then one can find a finite $Z$ as above.
\end{lem}

\begin{proof}
The proof goes exactly like that of Lemma 16 of \cite{R2}. We only note that the space $Z$ defined in the proof in \cite{R2} is a $\Delta$-metric space.
\end{proof}

\begin{defn}
A countable group $\Gamma$ has the \emph{RZ-property} (standing for Ribes--Zalesski\u{\i}) if any finite product $\Gamma_1\cdots\Gamma_n$ of finitely generated subgroups of $\Gamma$ is closed in the profinite topology. 
\end{defn}

It was proved by Ribes--Zalesski\u{\i} \cite{RZ} that countable free groups have the RZ-property. Moreover, they essentially showed in \cite{RZ} that, if $\Lambda\le \Gamma$ has finite index then $\Lambda$ has the RZ-property iff $\Gamma$ has the RZ-property. This gives the following fact we will need in our proof.

\begin{prop}\label{amalgamRZ}
Let $\Gamma_1,\Gamma_2$ be two finite groups, and $\Lambda$ be a common subgroup of $\Gamma_1,\Gamma_2$. Then $\Gamma_1*_{\Lambda}\Gamma_2$, the amalgamated free product of 
$\Gamma_1$ and $\Gamma_2$ over $\Lambda$, has the RZ-property.
\end{prop}

\begin{proof}
It is known that the amalgamated free product of finite groups is virtually free, i.e., it contains a free group as a subgroup of finite index (cf., e.g., Serre \cite{STrees}, Corollary to Proposition 11 on p. 120). By the above results of Ribes--Zalesski\u{\i} \cite{RZ}, virtually free groups have the RZ-property.
\end{proof}

Although not needed in our proof here, we note that Coulbois \cite{C} showed that the RZ-property is preserved under taking free products. The RZ-property will play an important role in our proof because of the following theorem, due to Rosendal.

\begin{thm}[Rosendal \cite{R}]\label{RosendalTheorem}
Let $\Delta$ be a countable distance value set. Let $\Gamma$ be a countable group with the RZ-property. Assume that $\pi \colon \Gamma \actson X$ is an isometric action of $\Gamma$ on a $\Delta$-metric space $X$. Then, for any finite $A \subseteq X$ and $F \subseteq \Gamma$ there exists a finite $\Delta$-metric space $Y$ containing $A$, and an isometric action $\pi' \colon \Gamma \actson Y$ such that for all $\gamma \in F$ and all $a \in A$ one has $\pi'(\gamma)a=\pi(\gamma)a$.
\end{thm}

Actually, Rosendal's theorem is an equivalence (the RZ-property is equivalent to the so-called finite approximability of actions on metric spaces; cf. \cite{R}) but we only need the implication mentioned above.

\begin{defn}
Let $\mcK_\Delta$ be the class of all structures $(X,G)$ such that 
\begin{itemize}
\item $X$ is a finite $\Delta$-metric space.
\item $G$ is a finite group.
\item $G$ acts isometrically on $X$.
\end{itemize}
\end{defn}

Note that we accept the case where $X$ is empty, considering that any group acts isometrically on the empty set. Also, $G$ is allowed to be equal to $\{1\}$.

\begin{thm}\label{KDeltaFraisse}
$\mcK_\Delta$ is a Fraïssé class.
\end{thm}

\begin{proof} The hereditary property is obvious. The joint embedding property is also easily witnessed by the product action. We only need to prove the amalgamation property. Assume that $X$ is a subspace contained in two finite $\Delta$-metric spaces $Y_1,Y_2$, and that $\Lambda$ is a  subgroup of two finite groups $\Gamma_1,\Gamma_2$. Assume further that $\Gamma_1 \actson Y_1$, $\Gamma_2 \actson Y_2$ isometrically, in such a way that $X$ is $\Lambda$-invariant for both actions, and the two $\Lambda$-actions coincide on $X$. Let $\Gamma=\Gamma_1*_{\Lambda}\Gamma_2$.

We first define a $\Delta$-metric space $Z_0$ amalgamating $Y_1$ and $Y_2$ over $X$, in such a way that $Z_0=Y_1 \cup Y_2$, $Y_1 \cap Y_2=X$ and the action of $\Lambda$ on $Y$ induced by the actions of $\Gamma_1,\Gamma_2$ on $Y_1,Y_2$ is by isometries. To ensure $Z_0$ is a $\Delta$-metric space, we only need to define, for $y_1\in Y_1\setminus{X}$ and $y_2\in Y_2\setminus{X}$, 
$$ d_{Z_0}(y_1, y_2)=\min\left\{ \inf_{x\in X} \{ d_{Y_1}(y_1, x)+d_{Y_2}(x, y_2)\}, \sup(\Delta)\right\}. $$ 
Then we define by induction an increasing sequence $\{Z_n\}_{n\geq 1}$ of $\Delta$-metric spaces, as well as $\Gamma_1$ actions on each $Z_{2n-1}$ and $\Gamma_2$ actions on each $Z_{2n}$, so that all actions are compatible with each other and with the original actions of $\Gamma_1$ on $Y_1$ and $\Gamma_2$ on $Y_2$. To define $Z_1$, apply Lemma \ref{l:extension} to the $\Gamma_1$ action on $Y_1$ and the $\Lambda$ action on $Z_0$. To define $Z_2$, apply Lemma \ref{l:extension} to the $\Gamma_2$ action on $Y_2$ and the induced $\Lambda$ action on $Z_1$. In general, obtain $Z_{2n+1}$ by applying Lemma~\ref{l:extension} to the $\Gamma_1$ action on $Z_{2n-1}$ and the induced $\Lambda$ action on $Z_{2n}$, and obtain $Z_{2n+2}$ by applying Lemma~\ref{l:extension} to the $\Gamma_2$ action on $Z_{2n}$ and the induced $\Lambda$ action on $Z_{2n+1}$. 
Let $Z_\infty=\bigcup_{n\geq 0}Z_n$. Then our construction gives an action of $\Gamma$ on $Z_\infty$. 

Now, using the fact that $\Gamma$ has the RZ-property from Proposition~\ref{amalgamRZ}, we apply Rosendal's theorem \ref{RosendalTheorem} to the action of $\Gamma$ on $Z_\infty$, with $A=Z_0$ and $F=\Gamma_1\cup \Gamma_2$, to find a finite $\Delta$-metric space $Y$ containing $Z_0$, and an isometric action $\Gamma$ on $Y$ that extends the original actions of $\Gamma_1,\Gamma_2$ on $Y_1,Y_2$ respectively. Then, let $G$ be the subgroup of $\Iso(Y)$ generated by $F=\Gamma_1\cup \Gamma_2$. The actions of elements of $F$ extends to an action of $G$ on $Y$, which gives an amalgam $(Y, G)\in \mathcal{K}_\Delta$ of $(Y_1, \Gamma_1)$ and $(Y_2,\Gamma_2)$ over $(X,\Lambda)$. 
\end{proof}

Denote the Fraïssé limit of $\mcK_\Delta$ by $(X_\Delta,H_\Delta)$, where 
$X_\Delta$ is a $\Delta$-metric space and $H_\Delta$ is a locally finite group 
acting isometrically on $X_\Delta$. The following lemmas will give us the main 
result of this section.

\begin{lem}
$X_\Delta$ is isometric to $\U_\Delta$.
\end{lem}

\begin{proof} We verify the Urysohn property from Proposition~\ref{UrysohnProperty} for $X_\Delta$. Let $A$ be a finite subset of $X_\Delta$, and $f: A\to \Delta$ satisfying
$|f(a)-f(b)|\leq d(a,b)\leq f(a)+f(b)$ for all $a, b\in A$. Then, viewing 
$A\cup\{f\}$ as a metric space, $(A,\{1\})\in\mcK_\Delta$ is a substructure 
embedded in $(A \cup \{f\},\{1\}) \in \mcK_\Delta$. By the universality and 
ultrahomogeneity of a Fraïssé limit we may find $x \in X_\Delta$ such that 
$d(x,a)=f(a)$ for all $a \in A$. 
\end{proof}

\begin{lem}
$H_\Delta$ acts faithfully on $X_\Delta$, i.e., if $g\in H_\Delta\setminus\{1\}$ there is $x\in X_\Delta$ such that $g\cdot x\neq x$. 
\end{lem}

\begin{proof} To see that $H_\Delta$ acts faithfully, let $g \in H_\Delta \setminus\{1\}$ and let $\Lambda$ be the finite subgroup of $H_\Delta$ generated by $g$. Let $c\in \Delta$ and let $X=(\Lambda,d)$ be a $\Delta$-metric space with $d(a, b)=c$ for any distinct $a, b\in \Lambda$. Then the left multiplication of $\Lambda$ is a faithful action of $\Lambda$ on $X$ by isometries. Thus $(X, \Lambda)\in \mcK_\Delta$. By the universality of $X_\Delta$, $X$ can be realized as a subset of $X_\Delta$. By the ultrahomogeneity of $(X_\Delta, H_\Delta)$,  the action of $\Lambda$ on $X_\Delta$ extends the $\Lambda$ action on $X$. Since the action of $\Lambda$ on $X$ is faithful, it follows that the action of $\Lambda$ on $X_\Delta$ is faithful.
\end{proof}

\begin{lem} $H_\Delta$ is isomorphic to $\mathbb{H}$.
\end{lem}
\begin{proof} We verify the property ({\sc E}) from Proposition~\ref{propertyE} for $H_\Delta$. 
Let $\Lambda$ be a finite subgroup of $H_\Delta$, and $i \colon \Lambda \to \Gamma$ be a group embedding of $\Lambda$ into a finite group $\Gamma$. 
Then $i$ induces an embedding from the structure $(\emptyset, \Lambda)\in\mcK_\Delta$ into $(\emptyset,\Gamma)\in\mcK_\Delta$. By the universality and ultrahomogeneity of $(X_\Delta, H_\Delta)$ we see that there is a group embedding $j \colon \Gamma \to H_\Delta$ such that 
$j \circ i(g)=g$ for all $g \in \Lambda$.
\end{proof}

\begin{lem}
$H_\Delta$ is dense in $\Iso(X_\Delta)$.
\end{lem}

\begin{proof}
Let $g \in \Iso(X_\Delta)$, and let $A$ be a finite subset of $X_\Delta$. We need to find an element of $H_\Delta$ coinciding with $g$ on $A$. Let $\Gamma$ be the subgroup of $\Iso(X_\Delta)$ generated by $g$. Then $\Gamma$ acts on $X_\Delta$ by isometries. We claim that there is a finite subset $B$ of $X_\Delta$ containing $A$ and an isometry $h$ of $B$ which coincides with $g$ on $A$. Indeed, if $\Gamma$ is finite, then we can let $B=\Gamma\cdot A$ and $h=g$. If $\Gamma$ is infinite, then it is a free group (isomorphic to $\mathbb{Z}$), and it has the RZ-property. We can then apply Rosendal's theorem \ref{RosendalTheorem} to find a finite $B$ containing $A$ and an isometry $h$ of $B$ which coincides with $g$ on $A$. 

Letting $H$ denote the finite group generated by $h$, we see that $(B,H) \in \mcK_\Delta$. We may realize the embedding from $(B,\{1\})$ into $(B,H)$ inside $(X_\Delta, H_\Delta)$, which gives us a finite subgroup of $H_\Delta$ isomorphic to $H$ and acting like $H$ on $B$. In particular there exists an element of $H_\Delta$ coinciding with $g$ on $A$.
\end{proof}

We have thus proved the following result. 

\begin{thm} \label{DeltaMain}
Let $\Delta$ be any countable distance value set. Then $\Iso(\U_\Delta)$ contains a dense subgroup that is isomorphic to $\mathbb{H}$.
\end{thm}
Theorem~\ref{DeltaMain} immediately gives (2), (3), (5) of Theorem~\ref{mainthm}. 

\begin{cor} $\Iso(\mathbb{QU})$, $\Aut(\mathcal{R})$, and $S_\infty$ all contain $\mathbb{H}$ as a dense subgroup.
\end{cor}

The other part of Vershik's conjecture, Theorem~\ref{mainthm} (1),  is a corollary of Theorem~\ref{mainthm} (2) from a standard argument.

\begin{thm}\label{IsoU}
$\Iso(\U)$ contains $\mathbb{H}$ as a dense subgroup.
\end{thm}

\begin{proof} Since $\mathbb{QU}$ is a countable dense subset of $\U$, the map $h\mapsto \overline{h}$ sending $h\in\Iso(\mathbb{QU})$ to its completion $\overline{h}\in \Iso(\U)$ is a well-defined group embedding. Since $\Iso(\mathbb{QU})$ has the permutation group topology and $\Iso(\U)$ has the pointwise convergence topology, this map is continuous. If $H$ is a dense subgroup of $\Iso(\mathbb{QU})$ then $\overline{H}=\{\overline{h}: h\in H\}$ is a dense subgroup of $\Iso(\U)$ isomorphic to $H$.
\end{proof}

To conclude this section, we note that the idea of considering finite metric spaces, with finite groups acting on them, as forming a Fra\"iss\'e class has already  been considered in Doucha's paper \cite{D} (though his formalism is different from ours). In particular, Theorem 0.2 in \cite{D} is closely related to the results in this section, though the proof is different. One may think of Theorem 0.2 of \cite{D} as a precursor to Theorem \ref{KDeltaFraisse}. We also note that our approach gives an answer to Question 3.7 from \cite{D}: it follows from our results that the group $\mathbb H$ appearing in \cite{D} is equal to $\Iso(\mathbb U)$.

\section{Automorphism Groups of Relational Structures\label{sec:structures}}
In this section we show some analogous results to the main results of the preceding section for ultraextensive relational structures. As a corollary, we will obtain Theorem~\ref{mainthm} (6). We first recall the concept of the HL-property defined in \cite{EG2} and develop some results necessary for our proof.

\subsection{The HL-property of a group}

\begin{defn}[Herwig--Lascar \cite{HL}]
Let $G$ be a group and let $H_1,\dots,H_n\leq G$. A \textit{left system} of equations on $H_1,\dots,H_n$ is a finite set of equations with variables $x_1,\dots,x_m$ and constants $g_1,\dots,g_l$, where each equation is of the form 
\[
x_iH_j=g_kH_j \text{ or } x_iH_j=x_rg_kH_j
\]
for $1\leq i,r\leq m$, $1\leq k\leq l$ and $1\leq j\leq n$. 
\end{defn} 

\begin{defn}
Let $G$ be a group. We say that $G$ has the \textit{HL-property} (standing for Herwig--Lascar) if for every finitely generated $H_1,\dots,H_n\leq G$ and left system $\Sigma$ of equations on $H_1,\dots,H_n$, if $\Sigma$ does not have a solution, then there exist normal subgroups of finite index $N_1,\dots,N_n\unlhd G$ such that, letting $\Sigma({\vec{N}})$ be the left system of equations obtained from $\Sigma$ by replacing all occurrences of $H_1, \dots, H_n$ by $N_1H_1,\dots,N_nH_n$ respectively, $\Sigma({\vec{N}})$ does not have a solution either.
\end{defn}

Herwig--Lascar \cite{HL} proved that the HL-property implies the RZ-property for groups. They also essentially showed in \cite{HL} that finitely generated free groups have the HL-property. As a strengthening of Coulbois's result on the preservation of the RZ-property under taking free products \cite{C}, it was shown in \cite{EG2} that the HL-property is also preserved under taking free products.



\begin{lem}
Let $G$ be a group and $H\leq G$ be a subgroup of finite index. Then $G$ has the HL-property iff $H$ has the HL-property.
\end{lem}
\begin{proof} Recall the fact that if $H\leq G$ is a subgroup of finite index, then $G$ has a normal subgroup $N$ of finite index such that $N\leq H$. Thus to prove the lemma, we may assume without loss of generality that $H$ is normal in $G$.

First assume $G$ has the HL-property. Let $\Sigma$ be a left system on finitely generated $H_1,\dots,H_n\leq H$. We claim that if $\Sigma$ has a solution in $G$ then it also has a solution in $H$. Let $V_0$ be the set of all variables $x$ such that $x$ appears in an equation in $\Sigma$ of the form $xH_j=gH_j$. Since the constants in $\Sigma$ are in $H$, any solution for an $x\in V_0$ must be in $H$. Now let $V_1$ be the set of all variables which appear in an equation in $\Sigma$ of the form $x_iH_j=x_rgH_j$, where at least one of $x_i$ and $x_r$ is in $V_0$. We see that any solution for an $x\in V_1$ must also be in $H$. Repeat this construction and define  $V_2, V_3$, etc. Since $\Sigma$ is finite, we obtain a maximal set of variables $V=V_0\cup V_1\cup V_2\cup\dots$ so that any solution for an $x\in V$ must be in $H$. Let $\Sigma'\subseteq \Sigma$ be the subsystem of all equations that contain (only) variables in $V$. Then the subsystem $\Sigma\setminus\Sigma'$ contains only equations of the form $x_iH_j=x_rgH_j$ where both $x_i, x_r\not\in V$. Now if $\Sigma$ has a solution in $G$, say $x_1=\gamma_1, \dots, x_m=\gamma_m$, then $\gamma_{i_1},\dots, \gamma_{i_k}$ are in $H$ where $V=\{x_{i_1},\dots, x_{i_k}\}$ and all the other variables are from the same coset of $H$. Let $gH$ be this coset. Then $x_{i_j}=\gamma_{i_j}$ for $j=1,2,\dots,k$ and $x_i=g^{-1}\gamma_i$ for $i\notin \{i_1,i_2,\dots,i_k\}$ is a solution of $\Sigma$ that consists of only elements in $H$. We have thus shown the claim. Now assume $\Sigma$ does not have a solution in $H$, then by the claim it does not have a solution in $G$. Since $G$ has the HL-property, there are $N_1, \dots, N_n\unlhd G$  of finite index such that $\Sigma({\vec{N}})$ on $N_1H_1, \dots, N_nH_n$ does not have a solution in $G$. Let $K_j=N_j\cap H$ for $1\leq j\leq n$. Then $K_j\unlhd H$ is of finite index, and $\Sigma({\vec{K}})$ on $K_1H_1,\dots, K_nH_n$ does not have a solution in $H$, since any solution of $\Sigma({\vec{K}})$ is also a solution of $\Sigma({\vec{N}})$.

For the converse, assume $H$ has the HL-property and $H\unlhd G$ is of finite index. Let $\Sigma$ be a left system on finitely generated $H_1, \dots, H_n\leq G$. Let $L_j=H_j\cap H$ for $1\leq j\leq n$. Then for each $j$, $L_j\leq H$ is finitely generated and has finite index in $H_j$. For each $j$, let $h_{j,1}L_j, \dots, h_{j, S_j}L_j$ enumerate all the left cosets of $L_j$ in $H_j$. Then each equation of the form $xH_j=ygH_j$ (here $y$ could be $1$ or a variable) is equivalent to $xL_j=ygh_{j,s}L_j$ for some $1\leq s\leq S_j$. Now consider the collection $\mathcal{S}$ of all left systems $\Sigma'$ where each $\Sigma'$ is obtained from $\Sigma$ by replacing each equation in $\Sigma$ of the form $xH_j=ygH_j$ by an equation of the form $xL_j=ygh_{j,s}L_j$ for some $1\leq s\leq S_j$. There are only finitely many left systems in $\mathcal{S}$, and $\Sigma$ has a solution in $G$ iff a $\Sigma'\in \mathcal{S}$ has a solution in $G$. To verify the HL-property for $G$, suppose $\Sigma$ does not have a solution in $G$. Then none of $\Sigma'\in\mathcal{S}$ has a solution. Assuming that the HL-property holds for $L_1, \dots, L_n$ for $G$, then for each $\Sigma'\in\mathcal{S}$, there exist normal subgroups $N^{\Sigma'}_1, \dots, N^{\Sigma'}_n\unlhd G$ of finite index such that $\Sigma'({\vec{N^{\Sigma'}}})$ does not have a solution. Let $N_j=\bigcap_{\Sigma'\in\mathcal{S}}N_j^{\Sigma'}$. Then $N_j\unlhd G$ is still of finite index, and for each $\Sigma'\in\mathcal{S}$, $\Sigma'({\vec{N}})$ still does not have a solution, since a solution for $\Sigma'({\vec{N}})$ would be a solution for $\Sigma'({\vec{N^{\Sigma'}}})$. This implies that $\Sigma({\vec{N}})$ does not have a solution, since a solution for it would be a solution of $\Sigma'({\vec{N}})$ for some $\Sigma'\in\mathcal{S}$. To finish the proof, it suffices to check that the HL-property holds for $L_1, \dots, L_n$ for $G$.

The above argument shows that it suffices to prove the HL-property for $H_1, \dots, H_n$ in $G$ when $H_1,\dots, H_n\leq H\unlhd G$, which we demonstrate below. Let $Hg_1,\dots, Hg_T$ be the right cosets of $H$ in $G$. Then $G=Hg_1\cup\dots\cup Hg_T$. First, suppose $\Sigma$ has a solution $x_1=\gamma_1,\dots, x_m=\gamma_m$ in $G$. Then there are $1\leq t_1, \dots, t_m\leq T$ such that $\gamma_i\in Hg_{t_i}$ for all $1\leq i\leq m$. If $x_lH_j=gH_j$ is an equation in $\Sigma$, then the solution $x_l=\gamma_l\in Hg_{t_l}$ witnesses that 
$$y_lg_{t_l}H_jg_{t_l}^{-1}=gg_{t_l}^{-1}g_{t_l}H_jg^{-1}_{t_l}$$
has a solution $y_l=\lambda_l=\gamma_lg_{t_l}^{-1}\in H$. If we let $H_j^l=g_{t_l}H_jg_{t_l}^{-1}$, then $H_j^l\leq H$ since $H$ is normal, and the above equation becomes $y_lH_j^l=(gg_{t_l}^{-1})H_j^l$. Similarly, if $x_lH_j=x_kgH_j$ is an equation in $\Sigma$, then the solution $x_l=\gamma_l, x_k=\gamma_k$ witnesses that
$$ y_lH_j^l=y_k(g_{t_k}gg_{t_l}^{-1})H_j^l $$
has a solution $y_l=\lambda_l=\gamma_lg_{t_l}^{-1}, 
y_k=\lambda_k=\gamma_kg_{t_k}^{-1}$ in $H$. Now for each $\vec{t}=(t_1, \dots, 
t_m)$ where $1\leq t_1, \dots, t_m\leq T$, we obtain a left system 
$\Sigma_{\vec{t}}$ from $\Sigma$ by replacing each equation in $\Sigma$ by an 
equation of the above form. Note that all the constants appeared in 
$\Sigma_{\vec{t}}$ are elements of $H$. Let $\mathcal{S}$ be the collection of 
all such $\Sigma_{\vec{t}}$. By our construction, $\Sigma$ has a solution in 
$G$ iff a $\Sigma_{\vec{t}}\in \mathcal{S}$ has a solution in $H$. To verify 
the HL-property for $H_1, \dots, H_n$ in $G$, suppose $\Sigma$ does not have a 
solution in $G$. Then none of $\Sigma_{\vec{t}}\in \mathcal{S}$ has a solution 
in $H$. By the HL-property of $H$, for each $\Sigma_{\vec{t}}\in\mathcal{S}$, 
there exist normal subgroups $N_1^{\Sigma_{\vec{t}}}, \dots, 
N_n^{\Sigma_{\vec{t}}}\leq H$ of finite index such that 
$\Sigma_{\vec{t}}({\vec{N^{\Sigma_t}}})$ does not have a solution in $H$. Let 
$N_j=\bigcap_{\Sigma_{\vec{t}}\in\mathcal{S}}N_j^{\Sigma_{\vec{t}}}$. Then 
$N_j\unlhd H$ is still of finite index, and for each 
$\Sigma_{\vec{t}}\in\mathcal{S}$, $\Sigma_{\vec{t}}(\vec{N})$ still does not 
have a solution in $H$. Now each $N_j$ is of finite index in $G$ since $H$ is 
of finite index in $G$. Let $M_j\unlhd G$ be of finite index such that $M_j\leq 
N_j$. It follows from our construction of $\mathcal{S}$ that $\Sigma(\vec{M})$ 
does not have a solution in $G$, since any solution of $\Sigma(\vec{M})$ would 
give rise to a solution for some $\Sigma_{\vec{t}}(\vec{N})$ where 
$\Sigma_{\vec{t}}\in\mathcal{S}$.
\end{proof}

Similar to Proposition~\ref{amalgamRZ} we obtain the following proposition from the above lemma and the Herwig--Lascar theorem on the HL-property of finitely generated free groups.

\begin{prop}\label{amalgamHL}
Let $\Gamma_1,\Gamma_2$ be two finite groups, and $\Lambda$ be a common subgroup of $\Gamma_1,\Gamma_2$. Then $\Gamma_1*_{\Lambda}\Gamma_2$ has the HL-property.
\end{prop}

\subsection{The Fraïssé class of actions by automorphisms} 

Let $\mathcal{L}$ be a finite relational language. If $C$ and $D$ are $\mathcal{L}$-structures, a {\it homomorphism} from $C$ to $D$ is a map $\pi: C\to D$ such that for every $n$-ary relation $R\in\mathcal{L}$ and every $a_1,\dots, a_n\in C$,
$$ R^C(a_1,\dots, a_n)\Rightarrow R^D(\pi(a_1),\dots,\pi(a_n)). $$
If $\mathcal{T}$ is a set of $\mathcal{L}$-structures and $D$ is an $L$-structure, then $D$ is \emph{$\mathcal{T}$-free} if there is no $C\in\mathcal{T}$ and homomorphism $\pi: C\to D$.

An $\mathcal{L}$-structure $C$ is called a \textit{Gaifman clique} if for every $a, b\in C$ there is a relation symbol $R\in \mathcal{L}$ with arity $m\geq 2$ and $c_1,\dots, c_m\in C$ with $a, b\in \{c_1, \dots, c_m\}$ and $R^C(c_1,\dots, c_m)$. It is clear that if $C$ is a Gaifman clique and $D$ is a homomorphic image of $C$ (i.e., there is a surjective homomorphism $\pi: C\to D$), then $D$ is also a Gaifman clique. Moreover, if $C$ is a finite Gaifman clique, then it has only finitely many homomorphic images, up to isomorphism. 

\begin{defn} Let $\mathcal{T}$ be a finite set of finite $\mathcal{L}$-structures each of which is a Gaifman clique. 
Let $\mcK$ be the class of all pairs $(M,G)$ such that 
\begin{itemize}
\item $M$ is a finite $\mathcal{T}$-free $\mathcal{L}$-structure,
\item $G$ is a finite group, and
\item $G$ acts on $X$ by isomorphisms.
\end{itemize}
\end{defn}

Suppose $\mathcal{T}$ is a finite set of finite $\mathcal{L}$-structures each 
of which is a Gaifman clique. Let $\tilde{\mathcal{T}}$ be the set of all 
homomorphic images of structures in $\mathcal{T}$. Then $\tilde{\mathcal{T}}$ 
is also a finite set of $\mathcal{L}$-structures each of which is a Gaifman 
clique. The class of all finite $\mathcal{T}$-free $\mathcal{L}$-structures 
coincides with the class of all finite $\mathcal{L}$-structures that do not 
allow an isomorphic embedding from any $\tilde{T}\in \tilde{\mathcal{T}}$, and 
by Lemma 4.5 of Siniora--Solecki \cite{S2}, this is a Fraïssé class closed 
under taking free amalgams.

The following characterization of the HL-property was proved in \cite{EG2} as an analog of Rosendal's theorem~\ref{RosendalTheorem}. 

\begin{thm}\label{HLproperty}
Let $G$ be a group. Then the following are equivalent:
\begin{enumerate}
\item[\rm (i)] $G$ has the HL-property;
\item[\rm (ii)] Let $\mathcal{L}$ be a finite relational language with unary relation symbols $S_1,\dots, S_n\in\mathcal{L}$. Let $\mathcal{T}$ be a finite set of finite $\mathcal{L}$-structures. Let $D$ be a $\mathcal{T}$-free $\mathcal{L}$-structure such that $\{S_1^D, \dots, S_n^D\}$ is a partition of the domain of $D$. Let $C$ be a finite substructure of $D$. Let $F$ be a finite subset of $G$. Suppose that $\pi: G\actson D$ is a faithful action by isomorphisms and that $\pi$ is transitive on each $S_i^D$ for $i=1,\dots, n$. Then there exists a finite $\mathcal{T}$-free $\mathcal{L}$-structure $D'$ extending $C$, and an action $\pi':G\actson D'$ by isomorphisms such that for all $\gamma\in F$ and $a\in C$ one has $\pi'(\gamma)a=\pi(\gamma)a$.
\item[\rm (iii)] Clause {\rm (ii)} with the additional assumption that every structure $T\in\mathcal{T}$ is a Gaifman clique.
\end{enumerate}
\end{thm}

In the following we also prove an analog of Rosendal's lemma \ref{l:extension}.

\begin{lem}
Let $\mathcal{L}$ be a finite relational language. Let $\Gamma$ be a group, $\Lambda\le \Gamma$ a subgroup. Assume that $M \subseteq N$ are $\mathcal{L}$-structures, and that $\Lambda \actson N$, $\Gamma \actson M$ are compatible actions by isomorphisms. Then there exists an $\mathcal{L}$-structure $P$ extending $N$, and an action $\Gamma \actson P$ by isomorphisms that is compatible with the $\Lambda$-action on $N$.

Moreover, if $\Gamma$ and $N$ are both finite then one can find a finite $P$ as above.
\end{lem}

\begin{proof}
Define an equivalence relation $\sim$ on $N\times \Gamma$ by $(a_1, g_1)\sim (a_2, g_2)$ iff 
$$\mbox{ ($g_2^{-1}g_1\in \Lambda$ and $g_2^{-1}g_1\cdot a_1=a_2$) or ($a_1, a_2\in M$ and $g_1\cdot a_1=g_2\cdot a_2$)}. $$
To see that $\sim$ is an equivalence relation, we only need to note that if $a_1, a_2\in M$, $g_1\cdot a_1=g_2\cdot a_2$, $g_3^{-1}g_2\in\Lambda$, and $g_3^{-1}g_2\cdot a_2=a_3$, then $a_3\in M$ and $g_2\cdot a_2=g_3\cdot a_3$, and thus $(a_1, g_1)\sim (a_3, g_3)$.

Let $P=N\times \Gamma/\sim$. Let $[a, g]$ denote the equivalence class $[(a, g)]_\sim$. For $R\in\mathcal{L}$ an $n$-ary relation symbol, define $R^P([a_1,g_1], \dots, [a_n, g_n])$ iff there are $b_1, \dots, b_n\in N$ and $g\in \Gamma$ such that $[a_1, g_1]=[b_1, g], \dots, [a_n, g_n]=[b_n, g]$, and $R^N(b_1, \dots, b_n)$. To see that this is well-defined, suppose $[a_1, g_1]=[b_1, g]=[c_1, h], \dots, [a_n, g_n]=[b_n, g]=[c_n, h]$. We need to show that $R^N(b_1, \dots, b_n)$ iff $R^N(c_1, \dots, c_n)$. In all cases we have $h^{-1}g\cdot b_1=c_1, \dots, h^{-1}g\cdot b_n=c_n$. Since both $\Gamma\actson M$ and $\Lambda\actson N$ are by isomorphisms, we have $R^N(b_1, \dots, b_n)$ iff $R^N(c_1, \dots, c_n)$.

Now it is easy to see that $a\mapsto [a, 1]$ is an isomorphic embedding of $N$ into $P$. Define $\Gamma\actson P$ by letting $g\cdot [a, h]=[a, gh]$. If $g\in \Lambda$ and $a\in N$, $g\cdot [a, 1]=[a, g]=[g\cdot a, 1]$. Thus this action is compatible with $\Lambda\actson N$.

It is also obvious that if $N$ and $\Gamma$ are finite then $P$ is finite.
\end{proof}

Now the proofs of Theorem~\ref{KDeltaFraisse} and the lemmas following it can be repeated to establish the following theorem.

\begin{thm}\label{Fraisse}
$\mathcal{K}$ is a Fra\"{i}ssé class. Furthermore, if $(N_\infty,H_\infty)$ is 
the Fra\"{i}ssé limit of $\mathcal{K}$, then $H_\infty\cong \mathbb{H}$, 
$H_\infty$ acts faithfully on $N_\infty$, and $H_\infty$ is dense in 
$\Aut(N_\infty)$.
\end{thm}

\begin{cor}
For all $n\geq 3$, the automorphism group of the Henson graph $H_n$ contains $\mathbb{H}$ as a dense subgroup.
\end{cor}

\section{Many Dense Locally Finite Subgroups\label{sec:many}}

\subsection{Omnigenous locally finite groups}
We define a concept of omnigenous groups and show that all countable omnigenous locally finite groups are embeddable as a dense subgroup of $\Iso(\U_\Delta)$. We will need the following extension lemma.

\begin{lem} \label{1extension}Let $\Delta$ be any countable distance value set. Let $X$ be a finite $\Delta$-metric space. Let $\Lambda\leq \Gamma$ be finite groups and $\pi: \Lambda\to \Iso(X)$ be an isomorphic embedding. Then there is a finite $\Delta$-metric space $Y$ extending $X$ and an isomorphic embedding $\pi': \Gamma\to \Iso(Y)$ such that for any $\gamma\in \Lambda$ and $x\in X$, $\pi'(\gamma)(x)=\pi(\gamma)(x)$.
\end{lem}

\begin{proof} Define a pseudometric $\delta$ on $X\times \Gamma$ by
$$ \delta((x_1, g_1), (x_2, g_2))=\left\{\begin{array}{ll} d_X(\pi(g_2^{-1}g_1)(x_1), x_2), &\mbox{ if $g_2^{-1}g_1\in\Lambda$,} \\ \diam(X), & \mbox{ otherwise.}\end{array}\right. $$
Define $(x_1, g_1)\sim (x_2, g_2)$ iff $\delta((x_1, g_1), (x_2, g_2))=0$. Then $\sim$ is an equivalence relation on $X\times \Gamma$. Let $Y=X\times\Gamma/\sim$. Then $\delta$ gives rise to a metric $d_Y([x_1, g_1], [x_2,g_2])=\delta((x_1, g_1), (x_2, g_2))$. $Y$ is obviously a finite $\Delta$-metric space.

It is easy to see that the map $x\mapsto [x,1]$ is an isometric embedding from $X$ into $Y$. For any $\gamma\in \Gamma$, $x\in X$ and $g\in \Gamma$, let $\pi'(\gamma)([x, g])=[x, \gamma g]$. Then $\pi': \Gamma\to \Iso(Y)$ is an isomorphic embedding. We check that for any $\gamma\in\Lambda$ and $x\in X$, 
$$ \pi'(\gamma)(x)=\pi'(\gamma)([x,1])=[x,\gamma]=[\pi(\gamma)(x), 1]=\pi(\gamma)(x). $$
\end{proof}


\begin{defn} Let $G$ be a group. We say that $G$ is \emph{omnigenous} if for every finite subgroup $G_1\leq G$, finite groups $\Gamma_1\leq \Gamma_2$ and group isomorphism $\Psi_1: G_1\cong \Gamma_1$, there is a finite subgroup $G_2\leq G$ with $G_1\leq G_2$ and an onto homomorphism $\Psi_2: G_2\to \Gamma_2$ such that $\Psi_2\upharpoonright G_1=\Psi_1$.
\end{defn}

If we strengthen the requirement on $\Psi_2$ to be an isomorphism, then this becomes the property ({\sc E}) from Proposition~\ref{propertyE}. Thus $\mathbb{H}$ is omnigenous. 

\begin{thm}\label{omnigenous} Let $H$ be a countable omnigenous locally finite group. Then for any countable distance value set $\Delta$, $\Iso(\U_\Delta)$ contains $H$ as a dense subgroup.
\end{thm}

\begin{proof}
Let $q_0, q_1, \dots$ be an enumeration of all partial isometries of $\U_\Delta$. Fix also an enumeration of all elements of $H$. We will define by induction infinite sequences of following objects:
\begin{itemize}
\item finite subsets $D_n$ of $\U_\Delta$, for $n\geq 1$;
\item elements $h_1, \dots, h_{2n}\in H$, and $H_n=\langle h_1,\dots, h_{2n}\rangle\leq H$, for $n\geq 1$;
\item group embeddings $\pi_n:H_n\to \Iso(D_n)$, for $n\geq 1$, 
\end{itemize}
such that
\begin{enumerate}
\item[(i)] for each $n\geq 0$, $q_n\subseteq \pi_{n+1}(h_{2n+2})$; in particular $\dom(q_n)\cup\rng(q_n)\subseteq D_{n+1}$;
\item[(ii)] for each $n\geq 1$, $D_n\subseteq D_{n+1}$;
\item[(iii)] for each $n\geq 1$, $g\in H_n$, and $x\in D_n$,  $\pi_{n+1}(g)(x)=\pi_n(g)(x)$;
\item[(iv)] for each $n\geq 0$, $h_{2n+1}$ is the least element of $H\setminus \{h_1, \dots, h_{2n}\}$ in the fixed enumeration of elements of $H$.
\end{enumerate}
Granting such sequences, it follows from (i) that $\bigcup_{n=1}^\infty D_n=\U_\Delta$.  From (ii) and (iii), it follows that for any $g\in H_n$, we have
$$ \pi_n(g)\subseteq \pi_{n+1}(g) \subseteq \cdots \subseteq \pi_{n+m}(g)\subseteq \cdots $$
and the limit $\lim_{m\to\infty}\pi_{n+m}(g)$ exists and is a full isometry of $\U_\Delta$ extending $\pi_n(g)$. Let $\Gamma_n=\pi_n(H_n)$ and let $i_n: \Gamma_n\to \Gamma_{n+1}$ be the isomorphic embedding with $i_n(\pi_n(g))=\pi_{n+1}(g)$ for all $g\in H_n$. We have a direct system
$$ \Gamma_1\stackrel{i_1}{\longrightarrow}\Gamma_2\stackrel{i_2}{\longrightarrow}\cdots $$
giving a direct limit $\Gamma$ that is a dense locally finite subgroup of $\Iso(\U_\Delta)$. We may regard the group embeddings as inclusions, and the direct limit of the system as an increasing union $\Gamma=\bigcup_{n=1}^\infty\Gamma_n$. Moreover, we have $\bigcup_{n=1}^\infty H_n\cong \Gamma$. By (iv), $\bigcup_{n=1}^\infty H_n=H$, and thus $H\cong \Gamma$. 

Assume that $D_n$, $h_1, \dots, h_{2n}$, and $\pi_n$ have been defined. We proceed to define $D_{n+1}$, $h_{2n+1}$, $h_{2n+2}$, and $\pi_{n+1}: H_{n+1}=\langle H_n, h_{2n+1}, h_{2n+2}\rangle\to \Iso(D_{n+1})$.

First, let $h_{2n+1}$ be the least element of $H\setminus \{h_1, \dots, 
h_{2n}\}$ in the fixed enumeration of elements of $H$. We define a 
$\Delta$-metric space $X$ extending $D_n$, and an isomorphic embedding 
$\sigma_n: \langle H_n, h_{2n+1}\rangle\to \Iso(X)$ such that for all $g\in 
H_n$ and $x\in D_n$, we have $\sigma_n(g)(x)=\pi_n(g)(x)$. If $n=0$, let $a$ be 
the order of $h_1$ and $c\in \Delta$ be arbitrary, and define $X$ to be a set 
with $a$ many elements with $d_X(x,y)=c$ iff $x\neq y\in X$. Define 
$\sigma_1:\langle h_1\rangle\to \Iso(X)$ by letting $\sigma_1(h_1)$ be a cyclic 
permutation on $X$. If $n\geq 1$, we check if $\pi_n: H_n\to \Iso(D_n)$ can be 
extended to an isomorphic embedding $\sigma_n: \langle H_n, h_{2n+1}\rangle\to 
\Iso(D_n)$. If so, then we let $X=D_n$ and $\sigma_n$ be such an extension. 
Assume that $\pi_n$ cannot be extended to an isomorphic embedding from $\langle 
H_n, h_{2n+1}\rangle$ to $\Iso(D_n)$. We apply Lemma~\ref{1extension} to obtain 
a finite $\Delta$-metric space $X$ extending $D_n$ and an isomorphic embedding 
$\sigma_n:\langle H_n, h_{2n+1}\rangle\to \Iso(X)$ such that for any $g\in H_n$ 
and $x\in D_n$, $\sigma_n(g)(x)=\pi_n(g)(x)$. 

Using the universality and ultrahomogeneity of $\U_\Delta$, we may assume the above $X$ are defined as a subset of $\U_\Delta$. Next we further extend $X$ to define $D_{n+1}$. Apply Theorem~\ref{SoleckiTheorem} to obtain a strongly coherent $S$-extension $(Y, \phi)$ of $X\cup \dom(q_n)\cup\rng(q_n)$. Now $\phi(q_n)$ is an element of $\Iso(Y)$ extending $q_n$. By Lemma~\ref{lem:strongcoherence} $\phi$ gives rise to an isomorphic embedding from $\Iso(X)$ to $\Iso(Y)$, which we still denote by $\phi:\Iso(X)\to \Iso(Y)$. 

Let $G_1=\langle H_n, h_{2n+1}\rangle\leq H$. Let $\Psi_1=\phi\circ \sigma_n$ be the isomorphic embedding from $G_1$ into $\Iso(Y)$. Let $\Lambda_1=\Psi_1(G_1)$ and $\Lambda_2=\langle \Lambda_1, \phi(q_n)\rangle\leq \Iso(Y)$. Since $H$ is omnigenous, there is a finite $G_2\leq H$ and an onto homomorphism $\Psi_2: G_2\to \Lambda_2$ such that $\Psi_2\upharpoonright H_n=\Psi_1=\phi\circ\sigma_n$. Let $h_{2n+2}\in \Psi_2^{-1}(\{\phi(q_n)\})$. Thus $\Psi_2(h_{2n+2})=\phi(q_n)$. By redefining, we can assume $G_2=\langle G_1, h_{2n+2}\rangle$.

Let $b=\diam(Y)$ and let $G_2$ be given the discrete metric with constant value $b$. Then $G_2$ becomes a $\Delta$-metric space. We define a finite $\Delta$-metric space $Z=Y\cup G_2$ to be the disjoint union of $Y$ and $G_2$, with $d_Z(y, g)=b$ for all $y\in Y$ and $g\in G_2$. Appealing again to the universality and ultrahomogeneity of $\U_\Delta$, we may assume that all of these extensions took place inside $\U_\Delta$. We let $D_{n+1}=Z\subseteq\U_\Delta$.

We have $H_{n+1}=\langle H_n, h_{2n+1}, h_{2n+2}\rangle=\langle G_1, h_{2n+2}\rangle=G_2$. 
Define $\pi_{n+1}: H_{n+1}\to \Iso(D_{n+1})=\Iso(Z)$ by letting $\pi_{n+1}(g)\upharpoonright Y=\Psi_2(g)$ and $\pi_{n+1}(g)(h)=gh$ for all $h\in G_2$. Then for any $g\in H_n$ and $x\in D_n$, $\pi_{n+1}(g)(x)=\Psi_1(g)(x)=\sigma_n(g)(x)=\pi_n(g)(x)$. To complete the proof, we need to verify that $\pi_{n+1}$ thus defined is a group isomorphism. For this, we show that for all 
$$g_1, \dots, g_k\in \{h_1, \dots, h_{2n}, h_{2n+1}, h_{2n+2}\} \mbox{ and } \epsilon_1, \dots, \epsilon_k\in\{+1,-1\}, $$
we have
$$ g_1^{\epsilon_1}\cdots g_k^{\epsilon_k}=1\iff \pi_{n+1}(g_1)^{\epsilon_1}\cdots \pi_{n+1}(g_k)^{\epsilon_k}=1. $$
First, suppose $g_1^{\epsilon_1}\cdots g_k^{\epsilon_k}=1$. Observe that, as an element of $\Iso(Z)$, the action of $\pi_{n+1}(g_1)^{\epsilon_1}\cdots \pi_{n+1}(g_k)^{\epsilon_k}$ on the $Y$ part of $Z$ is given by
$\Psi_2(g_1)^{\epsilon_1}\cdots \Psi_2(g_k)^{\epsilon_k}$. Since $\Psi_2$ is a homomorphism, we have $\Psi_2(g_1)^{\epsilon_1}\cdots \Psi_2(g_k)^{\epsilon_k}=1$. On the other hand, on the $G_2$ part of $Z$, the action of $\pi_{n+1}(g_1)^{\epsilon_1}\cdots \pi_{n+1}(g_k)^{\epsilon_k}$ is the same as the left multiplication by $g_1^{\epsilon_1}\cdots g_k^{\epsilon_k}=1$. Since both these actions are identity, we have $\pi_{n+1}(g_1)^{\epsilon_1}\cdots \pi_{n+1}(g_k)^{\epsilon_k}=1$. Conversely, if $\pi_{n+1}(g_1)^{\epsilon_1}\cdots \pi_{n+1}(g_k)^{\epsilon_k}=1$, then its action on the $G_2$ part is the left multiplication by 
$g_1^{\epsilon_1}\cdots g_k^{\epsilon_k}$; thus $g_1^{\epsilon_1}\cdots g_k^{\epsilon_k}=1$.
\end{proof}

\subsection{A family of omnigenous locally finite groups\label{omnigenousfamily}}
In this subsection we construct an uncountable family of pairwise non-isomorphic, omnigenous, universal countable locally finite groups. 

Let $P$ be a set of prime numbers. If $P\neq \emptyset$, enumerate its elements as $p_0<p_1<p_2<\dots$. Note that for each $n\geq 0$, there are infinitely many elements of order $p_n$ in $\mathbb{H}$. Fix a subset $S=\{s_0, s_1, s_2, \dots\}\subseteq \mathbb{H}$ where each $s_n$ is of order $p_n$, such that $\mathbb{H}\setminus S$ still generates a universal countable locally finite group. This is easy to arrange since $\mathbb{H}\oplus \mathbb{H}$ is embedded as a subgroup of $\mathbb{H}$ and we may choose $S$ as a subset of the first copy of $\mathbb{H}$.

Let $X$ be the disjoint union of $\mathbb{H}$ with a copy of $\mathbb{Z}_{p_n}=\mathbb{Z}/p_n\mathbb{Z}$ for each $n\geq 0$. For clarity, we will denote the copy of $\mathbb{H}$ as a subset of $X$ as $Y$, and for each $n\geq 0$, the copy of $\mathbb{Z}_{p_n}$ as a subset of $X$ as $Z_n$. Thus $X=Y\cup\bigcup_{n\geq 0}Z_n$. 

For $T\subseteq\mathbb{H}$, we say that $T$ is \emph{of type} $P$ if $T=\{t_0, t_1, t_2,\dots\}$ where each $t_n$ is of order $p_n\in P$. For any $T\subseteq\mathbb{H}$ of type $P$, we define a map $\lambda_T: \mathbb{H}\to \Sym(X)$ as follows. For all $g\in \mathbb{H}$, $\lambda_T(g)$ acts on $Y=\mathbb{H}$ as the left multiplication by $g$. If $g\not\in T$, then $\lambda_T(g)$ acts on each $Z_n$ as identity. If $g\in T$ and the order of $g$ is $p_n$, then $\lambda_T(g)$ acts on $Z_n=\mathbb{Z}_{p_n}$ as $+1$, and acts on other $Z_m$, $m\neq n$, as identity. Let $G_T$ be the subgroup of $\Sym(X)$ generated by the set $\lambda_T(\mathbb{H})$. Note that for any subset $A\subseteq \mathbb{H}\setminus T$, $\lambda_T$ gives an isomorphism between $\langle A\rangle\leq \mathbb{H}$ and $\langle \lambda_T(A)\rangle\leq \Sym(X)$. For each $t\in T$, $\lambda_T(t)$ has the same order as $t$. 

Also note that for any $B\subseteq\mathbb{H}$, the map $\lambda_T(b)\mapsto b$ induces a homomorphism from $\langle \lambda_T(B)\rangle$ onto $\langle B\rangle$. To see this, we need to show that for all 
$g_1, \dots, g_l\in B$ and $\epsilon_1, \dots, \epsilon_l\in\{+1,-1\}$, if $\lambda_T(g_1)^{\epsilon_1}\cdots \lambda_T(g_l)^{\epsilon_l}=1$, then $g_1^{\epsilon_1}\cdots g_l^{\epsilon_l}=1$. However, this is obvious by observing the action of $\lambda_T(g_1)^{\epsilon_1}\cdots \lambda_T(g_l)^{\epsilon_l}$ on $Y=\mathbb{H}$. 

We will construct a omnigenous, universal countable locally finite group $H_P$ as a direct limit of a direct system
$$ H_0\stackrel{e_1}{\longrightarrow} H_1\stackrel{e_2}{\longrightarrow} H_2\stackrel{e_3}{\longrightarrow}\cdots\cdots$$
where $e_k: H_{k-1}\to H_k$ is an isomorphic embedding for all $k\geq 1$. In fact, each $H_k$ will be of the form $G_{T_k}$ for some $T_k\subseteq \mathbb{H}$ of type $P$. 

We define the $H_k, e_k$ by induction on $k$. For $k=0$, let $T_0=S$ and $H_0=G_{T_0}$. Since $\mathbb{H}\setminus S$ generates a universal countable locally finite group, $H_0$ is universal for all countable locally finite groups. Since $H_0$ will be embedded as a subgroup of $H_P$, $H_P$ is thus universal as well. In the rest of our construction we focus on the omnigenous property of $H_P$.

In general, suppose $H_k=G_{T_k}$ has been defined. Let $i_{k+1}: H_k\to \mathbb{H}$ be an isomorphic embedding. Let $T_{k+1}=i_{k+1}(\lambda_{T_k}(T_k))$. Then $T_{k+1}$ is of type $P$. Let $H_{k+1}=G_{T_{k+1}}$. Define a map $f: \lambda_{T_k}(\mathbb{H})\to H_{k+1}$ by $f(\gamma)=\lambda_{T_{k+1}}(i_{k+1}(\gamma))$ for all $\gamma\in \lambda_{T_k}(\mathbb{H})$. We claim that $f$ extends uniquely to an isomorphic embedding from $H_k$ into $H_{k+1}$.
For this, we show that for all 
$$\gamma_1, \dots, \gamma_l\in \lambda_{T_k}(\mathbb{H}) \mbox{ and } \epsilon_1, \dots, \epsilon_l\in\{+1,-1\}, $$
we have
$$ \gamma_1^{\epsilon_1}\cdots \gamma_l^{\epsilon_l}=1\iff f(\gamma_1)^{\epsilon_1}\cdots f(\gamma_l)^{\epsilon_l}=1. $$
Suppose $\gamma_i=\lambda_{T_k}(g_i)$ for $g_i\in \mathbb{H}$ for all $1\leq i\leq l$. First, assume $\gamma_1^{\epsilon_1}\cdots \gamma_l^{\epsilon_l}=1$. Then, by observing the action of this element on the $Y$ part, we get that $g_1^{\epsilon_1}\cdots g_l^{\epsilon_l}=1$. By observing the action of this element on the $Z_n$ parts, we conclude that, if $t\in T_k$ of order $p_n$ appears as $g_i$ for $1\leq i\leq l$, then $N_t=\sum_{g_i=t}\epsilon_i$ is a multiple of $p_n$. It follows that $i_{k+1}(\gamma_1)^{\epsilon_1}\cdots i_{k+1}(\gamma_k)^{\epsilon_k}=1$, and consequently $f(\gamma_1)^{\epsilon_1}\cdots f(\gamma_k)^{\epsilon_k}$ acts on the $Y$ part as identity. Moreover, for any $t\in T_k$ of order $p_n$, letting $t'=f(\lambda_{T_k}(t))\in T_{k+1}$, then
$$N_{t'}=\sum_{f(\gamma_i)=t'}\epsilon_i=\sum_{g_i=t}\epsilon_i $$
is a multiple of $p_n$. Thus $f(\gamma_1)^{\epsilon_1}\cdots 
f(\gamma_l)^{\epsilon_l}$ acts on the $Z_n$ parts also as identity. Therefore 
$f(\gamma_1)^{\epsilon_1}\cdots f(\gamma_l)^{\epsilon_l}=1$. Conversely, 
suppose $f(\gamma_1)^{\epsilon_1}\cdots f(\gamma_l)^{\epsilon_l}=1$. Then by 
observing the action of this element on the $Y$ part, we get that 
$i_{k+1}(\gamma_1)^{\epsilon_1}\cdots i_{k+1}(\gamma_l)^{\epsilon_l}=1$. Thus 
$\gamma_1^{\epsilon_1}\cdots \gamma_l^{\epsilon_l}=1$. We have thus established 
the claim. From the claim, let $e_{k+1}: H_k\to H_{k+1}$ be the unique 
isomorphic embedding extending $f$.

This finishes our definition of the direct system. As usual, we view all $e_k$ as inclusions, and $H_P$ as an increasing union of $H_k$. 

We verify that $H_P$ is omnigenous. For this, let $G_1\leq H_P$ be a finite 
subgroup. Let $k$ be sufficiently large that $G_1\leq H_k$. Let $\Gamma_1\leq 
\Gamma_2$ be finite and $\Psi_1: G_1\cong\Gamma_1$. Now consider 
$i_{k+1}(G_1)\leq \mathbb{H}$ and note that $i_{k+1}\circ \Psi_1^{-1}$ is an 
isomorphic embedding from $\Gamma_1$ into $\mathbb{H}$ with image 
$i_{k+1}(G_1)$. By property ({\sc E}) from Proposition~\ref{propertyE} for 
$\mathbb{H}$, there is an isomorphic embedding $j:\Gamma_2\to \mathbb{H}$ 
extending $i_{k+1}\circ\Psi_1^{-1}$. Let $G_2$ be the group generated by 
$\lambda_{T_{k+1}}(j(\Gamma_2))$. As noted before, there is a homomorphism 
$\psi$ from $G_2$ onto $j(\Gamma_2)$ such that $\psi\upharpoonright_{G_1}$ is 
an isomorphism. Let $\Psi_2: G_2\to \Gamma_2$ be $j^{-1}\circ \psi$. It is 
straightforward to check that $\Psi_2\upharpoonright e_{k+1}(G_1)=\Psi_1\circ 
e_{k+1}^{-1}$. This shows that $H_P$ is omnigenous.

The next lemma characterizes the isomorphism type of $H_P$ in terms of the set $P$.

\begin{lem}\label{HP} $P$ is exactly the set of all primes $p$ such that there are order-$p$ elements $\alpha, \beta\in H_P$ that are not conjugate in $H_P$.
\end{lem}

\begin{proof} First let $p_n\in P$. Let $a\in \mathbb{H}$ be an element of order $p_n$ such that $a\neq s_n\in S=T_0$. We claim that $\alpha=\lambda_S(a)\in H_0$ and $\beta=\lambda_S(s_n)\in H_0$ are not conjugate in $H_P$. Toward a contradiction, assume $\alpha, \beta$ are conjugate in $H_P$. Then there is $k\geq 0$ such that they are conjugate in $H_k$. By the construction of $H_k$, we have $\alpha=\lambda_{T_k}(g)$ for some $g\in \mathbb{H}\setminus T_k$ and $\beta=\lambda_{T_k}(h)$ for some $h\in T_k$ of order $p_n$. Let $g_1, \dots, g_l\in \mathbb{H}$ and $\epsilon_1, \dots, \epsilon_l\in\{+1,-1\}$ such that 
$$\lambda_{T_k}(g_1)^{\epsilon_1}\cdots \lambda_{T_k}(g_l)^{\epsilon_l}\lambda_{T_k}(g)\lambda_{T_k}(g_l)^{-\epsilon_l}\cdots \lambda_{T_k}(g_1)^{-\epsilon_1}=\lambda_{T_k}(h). $$
The action of the element on the left hand side on $Z_n$ is identity, while the action of $\lambda_{T_k}(h)$ on $Z_n$ is $+1$,  a contradiction.

On the other hand, suppose $\alpha, \beta\in H_P$ both have order $p\not\in P$. Then there is $k\geq 0$ such that $\alpha=\lambda_{T_k}(g)$ and $\beta=\lambda_{T_k}(h)$ for $g, h\in \mathbb{H}\setminus T_k$. By Proposition~\ref{cong} $g, h$ are conjugate in $\mathbb{H}$, i.e., there is $g_1\in \mathbb{H}$ such that $g_1gg_1^{-1}=h$. Then we claim 
$\lambda_{T_k}(g_1)\lambda_{T_k}(g)\lambda_{T_k}(g_1)^{-1}=\lambda_{T_k}(h)$. This is because, the action of the element on the left hand side on $Y$ is by left multiplication of $g_1gg_1^{-1}$, while the action of $\lambda_{T_k}(h)$ on $Y$ is by left multiplication of $h$, which are the same; on the other hand, the action of the element on the left hand side on all $Z_n$ is identity regardless of whether $g_1\in T_k$, which is the same as the action of $\lambda_{T_k}(h)$ on $Z_n$. Thus the claim holds true, and $\alpha$ and $\beta$ are conjugate in $H_k$. 
\end{proof}

By Lemma~\ref{HP}, if $P\neq P'$ are distinct sets of primes, then $H_P$ and $H_{P'}$ are not isomorphic. Since all $H_P$ are omnigenous, by Theorem~\ref{omnigenous} we can embed $H_P$ into $\Iso(\U_\Delta)$ as a dense subgroup. We have thus proved the following theorem.

\begin{thm}
There are continuum many pairwise nonisomorphic countable universal locally finite groups each of which can be embedded into $\Iso(\U_\Delta)$ as a dense subgroup.
\end{thm}

When $P=\emptyset$, it is easy to see that $H_P$ has the property ({\sc E}) from Proposition~\ref{propertyE}, and hence is isomorphic to $\mathbb{H}$. Thus we obtain another proof of Theorem~\ref{DeltaMain}.

\subsection{Ultrametric spaces}
In this subsection we deal with ultrametric Urysohn spaces and their isometry groups. We first recall some basic facts about ultrametric Urysohn spaces (cf., e.g., \cite{GS} and \cite{NVT}).

Recall that an \emph{ultrametric} $d$ on a space $X$ is a metric such that $$d(x, y)\leq \max\{d(x, z), d(y, z)\}$$ for all $x, y, z\in X$. In any separable ultrametric space, the ultrametric can take only countably many values. Consequently, there is no separable ultrametric space that is universal for all separable ultrametric spaces. 

Given any countable set $R$ of positive real numbers, an \emph{$R$-ultrametric 
space} is an ultrametric space in which the ultrametric takes positive values 
only in $R$. The class of 
all finite $R$-ultrametric spaces is a Fraïssé class, and we let 
$\mathbb{K}^u_R$ denote its Fraïssé limit. $\mathbb{K}^u_R$ is a universal 
countable, ultrahomogeneous $R$-ultrametric space, and we call it the 
\emph{universal countable $R$-ultrametric Urysohn space}. We endow 
$\Iso(\mathbb{K}^u_R)$ with the permutation group topology. 

Consider the completion of $\mathbb{K}^u_R$ under the pointwise convergence topology, which we denote as $\U^u_R$ and call the \emph{$R$-ultrametric Urysohn space}. $\U^u_R$ is a Polish $R$-ultrametric space which is universal for all Polish $R$-ultrametric spaces and is itself ultrahomogeneous. We endow $\Iso(\U^u_R)$ with the pointwise convergence topology.

By the standard argument in the proof of Theorem~\ref{IsoU}, any dense subgroup of $\Iso(\mathbb{K}^u_R)$ gives rise to an isomorphic dense subgroup of $\Iso(\U^u_R)$. We prove the following theorem.

\begin{thm}
For any non-empty countable set $R$ of positive real numbers, the following hold:
\begin{enumerate}
\item[(i)] $\Iso(\mathbb{K}^u_R)$ and $\Iso(\U^u_R)$ contain $\mathbb{H}$ as a dense subgroup. 
\item[(ii)] Every countable omnigenous locally finite group can be embedded into $\Iso(\mathbb{K}^u_R)$ or $\Iso(\U^u_R)$ as a dense subgroup.
\item[(iii)] There are continuum many non-isomorphic universal countable 
locally finite groups that can be embedded into each of $\Iso(\mathbb{K}^u_R)$ 
and $\Iso(\U^u_R)$ as a dense subgroup.
\end{enumerate}
\end{thm}

Our plan is to repeat the proof in the preceding subsections for $\Iso(\mathbb{K}^u_R)$. We first need a lemma for $R$-ultrametric spaces that is analogous to Lemma~\ref{1extension}. It turns out that the proof is verbatim the same as that of Lemma~\ref{1extension}, only noting that the pseudometric defined there is indeed a pseudo-ultrametric. We state the lemma below without proof.

\begin{lem}\label{1extensionultra} Let $R$ be any nonempty countable set of positive numbers. Let $X$ be a finite $R$-ultrametric space. Let $\lambda\leq \Gamma$ be finite groups and $\pi: \Lambda\to \Iso(X)$ be an isomorphic embedding. Then there is a finite $R$-ultrametric space $Y$ extending $X$ and an isomorphic embedding $\pi':\Gamma\to \Iso(Y)$ such that for any $\gamma\in \Lambda$ and $x\in X$, $\pi'(\gamma)(x)=\pi(\gamma)(x)$.
\end{lem}

The next thing we need is a result analogous to Solecki's Theorem~\ref{SoleckiTheorem} for $R$-ultrametric spaces. This is an easy consequence of the techniques used to prove the metric case (cf., e.g. \cite{S2}), but we give a self-contained proof here.

\begin{lem}\label{strongcoherenceultra}
Let $R$ be any nonempty countable set of positive real numbers. Let $X$ be a finite $R$-ultrametric space. Then $X$ has a strongly coherent S-extension $(Y,\phi)$ where $Y$ is a finite $R$-ultrametric space. 
\end{lem}
\begin{proof}
Let $D(X)=\{d_X(x, y)\,:\, x\neq y\in X\}$. We prove this by induction on $|D(X)|$. 

First consider the case $|D(X)|=1$. In this case $\Iso(X)=\Sym(X)$. Fix a linear order $<$ on $X$. Given any partial isometry (permutation) $p$ of $X$, define $\phi(p)$ to be the extension of $p$ by the (unique) $<$-order-preserving bijection between $X\setminus\dom(p)$ and $X\setminus\rng(p)$. Then $(X, \phi)$ is easily seen to be a strongly coherent S-extension of $X$.  

Suppose $|D(X)|>1$ and let $r$ be the least element of $D(X)$. For each $x\in X$ let $$B_r(x)=\{y\in X\,:\, d_X(x,y)\leq r\}=\{x\}\cup\{y\in X\,:\, d_X(x,y)=r\}.$$ Define $X_1=\{B_r(x)\,:\, x\in X\}$ and $d_1$ on $X_1$ by 
$$d_1(B_r(x_1), B_r(x_2))=\left\{\begin{array}{ll} d_X(x_1, x_2), &\mbox{ if $d_X(x_1, x_2)>r$,}\\
0, &\mbox{ otherwise.}\end{array}\right. $$
Then $|D(X_1)|=|D(X)|-1$. By the inductive hypothesis applied to $R=D(X)\setminus\{r\}$, $X_1$ has a finite $R$-ultrametric strongly coherent S-extension $(Y_1,\phi_1)$, where $D(Y_1)\subseteq R$. 

Let $N=\max\{|B_r(x)|\,:\, x\in X\}$ and fix an $x_0\in X$ such that $|B_r(x_0)|=N$. Fix a linear order $<_x$ on each $B_r(x)$; however, make $<_x$ depend only on $B_r(x)$ but not on $x$. Let $X_2=B_r(x_0)$ and $d_2=d_X$ on $X_2\subseteq X$. For each of $B_r(x)$, let $e_x: B_r(x)\to X_2$ be the order-preserving embedding so that $e_x(B_r(x))$ is an initial segment in $X_2$. Every $B_r(x)$ is identified as a subset of $X_2$ via $e_x$. We will view $e_x$ as an inclusion. For each partial isometry $p$ of $B_r(x)$ (viewed also as a partial isometry of $X_2$), let $\phi_x(p)$ be the extension of $p$ by the order-preserving bijection between $X_2\setminus \dom(p)$ and $X_2\setminus \rng(p)$. 

Let $Y=Y_1\times X_2$ and define $d_Y((u_1, u_2), (v_1, v_2))=\max\{ d_1(u_1, v_1), d_2(u_2, v_2)\}$. Every $x\in X$ is identified with $(B_r(x),e_x(x))\in Y$. If $p$ is a partial isometry of $X$, then $p$ induces a partial isometry of $X_1$, which we denote by $p_1$. For every $x\in \dom(p)\subseteq X$, $p$ induces a partial isometry between $B_r(x)$ and $B_r(p(x))$, which is identified as a partial isometry of $X_2$ via $e_x$ and $e_{p(x)}$, which we denote by $p_x$. Note that $p_x$ depends on $B_r(x)$ but not on $x$. Define $\phi(p)\in \Iso(Y)$ by
$$ \phi(p)(u, v)=\left\{\begin{array}{ll}(\phi_1(p_1)(u), \phi_{x}(p_x)(u,v)),   &  \mbox{ if $x\in \dom(p)$ and $d_Y(x, (u,v))\leq r$,} \\
(\phi_1(p_1)(u), v), & \mbox{ if there is no such $x$.} \end{array}\right.
$$
Then it is straightforward to check that $(Y, \phi)$ is a strongly coherent S-extension of $X$.
\end{proof}

The rest of the proof in the preceding section works verbatim. In particular, we obtain  the  space $X$ using Lemma~\ref{1extensionultra} and $Y$ using Lemma~\ref{strongcoherenceultra}, and then the space $Z$ constructed is an ultrametric space.

\section{Properties of Dense Locally Finite Subgroups\label{sec:all}}
In this section we study some properties of all dense locally finite subgroups of $\Iso(\U_\Delta)$ from the point of view of model theory and combinatorial group theory.

\subsection{Discerning types and discerning structures}
We first define some concepts and fix some notation. Throughout this section we 
let $\sL$ be a countable relational language (with equality). Given an 
$\sL$-structure $M$ and $A \subseteq M$, we denote by $\sL_A$ the language 
$\sL$ expanded by a constant symbol for each element of $A$.

\begin{defn}
Fix an $\sL$-structure $M$ and $A \subseteq M$. A \emph{$1$-type over $A$} is a set $p$ of $\sL_A$-formulas, with (at most) one free variable $x$, for which there exists 
$m \in M$ such that
$$p=\{\varphi(x) \colon M \models \varphi(m)\}. $$
Such $m$ is called a \emph{realization} of $p$ in $M$. We say that $p$ is \emph{nontrivial} if $p$ has a realization not belonging to $A$. We denote by $S_1(A)$ the set of all $1$-types over $A$.
\end{defn}

Our terminology differs slightly from the usual definition in that these types are usually referred to as types realized in $M$. In our context there will always be an underlying structure $M$, and any $1$-type over $A$ is realized in $M$. Note that a $1$-type over $A$ is nontrivial iff all realizations of $p$ do not belong to $A$, because if $a\in A$ and $b\in M\setminus A$, then they cannot have the same $1$-type over $A$ (one satisfies $x=a$ and the other one does not).


\begin{defn}
Let $M$ be a countable $\sL$-structure. Given $A \subseteq M$ and $p$ a $1$-type over $A$, we denote by $[p]$ the set of all realizations of $p$ in $M$, i.e., $[p]=\{m\in M: \forall \varphi\in p\ M\models \varphi(m)\}$. 

We say that $p$ is \emph{algebraic over} $A$ if $[p]$ is finite.
\end{defn}

As an example, consider the structure $\U_\Delta$, where $\Delta$ is a countable distance value set. As illustrated in Subsection~\ref{subsectionDelta}, $\U_\Delta$ can be viewed as a relational structure in a language $\sL$ with a binary relation symbol for each value of $\Delta$. If $A \subset \U_\Delta$ is finite, the $1$-type over $A$ of $x \in \U_\Delta$ is entirely determined by the distance function $a \mapsto d(a,x)$ (by ultrahomogeneity). Thus we may identify $1$-types over $A$ with \emph{Kat\v{e}tov maps} over $A$, i.e.~maps $f \colon A \to \Delta \cup\{0\}$ such that
$$\forall a,b \in A\ \ |f(a)-f(b)| \le d(a,b) \le f(a)+f(b). $$
We denote by $E(A)$ the set of all Kat\v{e}tov maps over $A$. Note that if $f \in E(A)$ is such that $f(a)=0$ for some $a$, then $f=d(a,\cdot)$. Below we will sometimes identify $A$ with the subset of $E(A)$ made up of trivial $1$-types.


\begin{defn}
Let $M$ be a countable $\sL$-structure, $A$ a finite substructure and $p$ a nontrivial $1$-type over $A$. We say that $p$ is \emph{discerning} if for every non-identity $g \in \Aut(M)$ there exists $x \in [p]$ such that $g(x) \ne x$. 
\end{defn}


Note that, if $A$ is finite and $p$ is both discerning and algebraic over $A$, then $\{g\in \Aut(M) \colon \forall x \in [p] \ g(x)=x\}=\{1\}$, hence $\Aut(M)$ is discrete. Equivalently, if $\Aut(M)$ is non-discrete, $A$ is finite and $p \in S_1(A)$ is discerning, then $p$ cannot be algebraic, thus has infinitely many realizations. 

\begin{defn}
Let $M$ be a countable $\sL$-structure. We say that $M$ is \emph{discerning} if for any finite $A\subseteq M$, every nontrivial $1$-type over $A$ is discerning.
\end{defn}

Note that a countable infinite set, viewed as a first-order structure with only the equality symbol in its language, is not discerning, because of the existence of elements of $\Aut(M)$ with finite support. 

\begin{defn}
Let $M$ be a countable ultrahomogeneous $\sL$-structure. We say that $M$ has \emph{rich types} if it satisfies the following conditions:
\begin{itemize}
\item For any finite $A\subseteq M$, and any $g \in \Aut(M) \setminus \{1\}$, there exists $x \not \in A$ such that $g(x) \ne x$;
\item For any finite $A\subseteq M$, any $x \ne y \in M$ with $x \not \in A$, and any nontrivial $p \in S_1(A)$, there exist $z \in M$ which is a realization of $p$ and an $\sL$-formula $\varphi$ such that $M\models \varphi(x,z)$ and $M\models \neg \varphi(y,z)$.
\end{itemize}
\end{defn}

Observe that the first condition above forbids the existence of nontrivial elements of $\Aut(M)$ with finite support.

\begin{prop}
Let $M$ be a countable ultrahomogeneous $\sL$-structure. If $M$ has rich types, then $M$ is discerning.
\end{prop}

\begin{proof}
Let $A\subseteq M$ be finite, $p$ a nontrivial type over $A$, and $g \in \Aut(M) \setminus \{1\}$. By assumption, $g$ cannot fix every element in the complement of $A$. So we pick $x \in M \setminus A$ such that $g(x) \ne x$. Apply the fact that $M$ has rich types to $p$, $x$, and $y =g(x)$. We thus obtain a realization $z$ of $p$ and an $\sL$-formula $\varphi$ such that $M\models \varphi(x, z)$ and $M\models \neg \varphi(y, z)$. Since $g\in \Aut(M)$, we have $M\models \varphi(y, g(z))$, and thus $g(z)\neq z$. This $z$ witnesses that $p$ is discerning.
\end{proof}

The random graph $\mathcal{R}$ has rich types. In particular, for $\Delta=\{1,2\}$, $\U_\Delta$ has rich types. In the above we already saw that $\U_{\{1\}}$ does not have rich types because it is not discerning. In general, $\U_\Delta$ can also fail to have rich types for the reason that it does not satisfy the second condition of the definition: it is possible that the type of $z$ over $A$ forces that $(x,z)$ and $(y,z)$ satisfy the same $\sL$-formulas. Nevertheless, this happens for fairly few types, so that one can still prove the following.

\begin{prop}\label{UDeltadiscerning}
Let $\Delta$ be a countable distance value set with $|\Delta| \ge 2$. Then $\U_\Delta$ is discerning.
\end{prop}

\begin{proof}
We first note that when $|\Delta|\geq 2$, $\U_\Delta$ does satisfy the first condition of the definition of rich types. In fact, let $A\subseteq \U_\Delta$ be finite and $g\in \Iso(\U_\Delta)\setminus \{1\}$. If $g$ fixes all elements of $A$, then there is $x\in \U_\Delta\setminus A$ with $g(x)\neq x$, since $g$ is not identity. Otherwise, let $a\in A$ be such that $g(a)\neq a$. If $g(a)\not\in A$, we can let $x=g(a)$, and then $g(x)\neq x$ since $d(a, x)=d(g(a), g(x))=d(x, g(x))$. Suppose $g(a)\in A$ and let $r=d(a, g(a))$. Since $|\Delta|\geq 2$, there is an $s\in\Delta\setminus\{r\}$. Suppose first $s<r$. Then let $k$ be the positive integer with $ks< r\leq (k+1)s$. Define a Kat\u{e}tov map $f$ over $A$ by letting 
$$ f(c)=\min\{ d(a, c)+ks,\ d(g(a), c)+r,\ \sup(\Delta)\} \mbox{ for all $c\in A$.}$$
In particular $f(a)=ks\neq r=f(g(a))$. Let $x\in \U_\Delta$ be such that $f(c)=d(x, c)$ for all $c\in A$. Then $x\not\in A$ since $d(x, c)=f(c)>0$ for every $c\in A$. Then $d(g(x), g(a))=d(x, a)=f(a)=ks\neq r=f(g(a))=d(x, g(a))$. Thus $g(x)\neq x$. Assume next $r<s$. Let $k$ be the positive integer with $kr<s\leq (k+1)r$. Define a Kat\u{e}tov map $f$ over $A$ by letting
$$ f(c)=\min\{ d(a,c)+kr,\ d(g(a), c)+s,\ \sup(\Delta)\} \mbox{ for all $c\in A$.} $$
Then by a similar argument we can find $x\not\in A$ with $g(x)\neq x$.

To prove $\U_\Delta$ is discerning, fix a finite set $A\subseteq \U_\Delta$ and nontrivial $f \in E(A)$. Then $f(c)>0$ for all $c\in A$. If $A=\emptyset$ the statement is obvious. So we assume $A\neq\emptyset$. Fix a non-identity $g \in \Iso(\U_\Delta)$ and $x\not\in A\cup g^{-1}(A)\cup g^{-2}(A)$ such that $x \ne g(x)$. Then $x, g(x)\not\in A\cup g^{-1}(A)$.

Let $B=A\cup g^{-1}(A)\cup\{x, g(x)\}$. Let $R=\max\{ f(c), d(a, b): c\in A, a\neq b\in B\}$ and $r=\min\{ f(c), d(a, b): c\in A, a\neq b\in B\}>0$. 

First assume $R>r$. Find $t\in\Delta$ such that $\frac{1}{2}R\leq \max\{r, R-r\}\leq t<R$. If $r\geq \frac{1}{2}R$ we can just let $t=r$. Otherwise, let $k$ be the positive integer with $kr<R\leq (k+1)r$ and let $t=kr$. Define a Kat\u{e}tov map $F$ over $B$ by letting
$$ F(b)=\left\{\begin{array}{ll} t, & \mbox{ if $b=g(x)$,}\\ R, & \mbox{ otherwise.}\end{array}\right. $$
Let $y\in \U_\Delta$ be such that $d(y, b)=F(b)$ for all $b\in B$. Since $d(y, x)=R$ and $d(y, g(x))=t$, we get that $g(y)\neq y$. Moreover, for any $c\in A$,  since $d(y, g^{-1}(c))=R$, we have $d(g(y), c)=R$. Finally, define a Kat\u{e}tov map $\hat{f}$ over $A\cup\{y, g(y)\}$ that is an extension of $f$ by letting $\hat{f}(y)=R$ and $\hat{f}(g(y))=t$. Let $z\in \U_\Delta$ be a realization of $\hat{f}$. Then $z$ is a realization of $f$, and $d(z, y)=R\neq t=d(z, g(y))$. Thus $g(z)\neq z$.
\end{proof}

\subsection{Mixed identities and MIF properties}\label{sec: infty mif}

In this subsection we consider mixed identities in a group and show that all 
dense subgroups of $\Iso(\U_\Delta)$ do not satisfy any nontrivial mixed 
identities, i.e., they are MIF (mixed identity free). For locally finite 
groups, we also define the more general notion of $\infty$-MIF and show that it 
coincides with omnigenity.

We refer the reader to Hull--Osin \cite{HO} (particularly Section 5 there) for 
an account of the study of mixed identities in group theory and for many 
existing results about the notion of MIF groups. The argument used to prove 
Theorem \ref{TheoremMIF} below is an adaptation to our context of arguments of 
Theorem 5.9 of \cite{HO} about \emph{highly transitive groups} 
(dense subgroups of the permutation group $S_\infty$ of an infinite countable 
set). 

\begin{defn}\label{def:MIF} Let $G$ be a group. Let $F_n$ be the free group 
generated by variables $x_1, \dots, x_n$. A \emph{nontrivial mixed identity} in 
$G$ is a word  $w(x_1, \dots, x_n) \in G \ast F_n \setminus G$ such that  
$w(g_1, \dots, g_n)=1$ for all $g_1, \dots, g_n\in G$. If there is no 
nontrivial mixed identity in $G$, we say that $G$ is \emph{mixed identity free} 
(\emph{MIF}).
\end{defn}

As noted in \cite{HO} (Remark 5.1), a group $G$ is MIF iff there is no nontrivial mixed identity $w(x)\in G*\langle x\rangle\setminus G$, as there is an isomorphism from $G*F_n$ into $G*\langle x\rangle$ that sends $x_1, \dots, x_n$ to $xgx, \dots, x^ngx^n$ for a $g\in G\setminus \{1\}$. In the sequel we only consider mixed identities with one free variable $x$. Such identities are of the form $g_1x^{m_1}g_2x^{m_2}\cdots g_{k}x^{m_{k}}$ for $g_1, \dots, g_k\in G$, $g_2, \dots, g_{k}\in G\setminus\{1\}$, and $m_1, \dots, m_{k}\in \mathbb{Z}\setminus\{0\}$. We denote such identities by $w(x; g_1, \dots, g_k)$.

\begin{prop}\label{subgroupMIF}
Let $G$ be a topological group. If $G$ is MIF then any dense subgroup of $G$ is MIF.
\end{prop}

\begin{proof} 
Let $\Gamma$ be a dense subgroup of $G$, and assume that $\Gamma$ satisfies a nontrivial mixed identify $w(\gamma;\gamma_1,\ldots,\gamma_k)=1$ for some fixed $\gamma_1,\ldots,\gamma_k \in \Gamma$ and all $\gamma \in \Gamma$. 
Since $\{g \in G \colon w(g;\gamma_1,\ldots,\gamma_k)=1\}$ is closed, and contains $\Gamma$, we conclude that we have $w(g;\gamma_1,\ldots,\gamma_k)=1$ for all $g\in G$. 
\end{proof}

\begin{thm}\label{TheoremMIF} Let $\sL$ be a countable relational language and $M$ be a countable ultrahomogeneous $\sL$-structure. Suppose $M$ is discerning and $\Aut(M)$ is non-discrete. Then for every $w(x) \in \Aut(M) \ast \langle x\rangle\setminus \Aut(M)$ the set
$\{g \in G \colon w(g) \ne 1\}$ is dense. In particular, $\Aut(M)$ is MIF, and so is any dense subgroup of $\Aut(M)$.
\end{thm}
\begin{proof}
Let $g_1,...,g_k\in \Aut(M)$ with $g_2,...,g_k\neq 1$, and $m_1,...,m_{k}\in \mathbb{Z}\setminus \{0\}$. 
We have to prove that the set consisting of all elements $g\in \Aut(M)$ such that $g_1g^{m_1}\cdots g_kg^{m_k}\neq 1$ is dense. We will use the following lemmas.

\begin{lem}\label{type1} Let $A, B$ be finite subsets of $M$ and $t$ be a nontrivial 1-type over $A$. Then there is $b\not\in B$ which is a realization of $t$.
\end{lem}

\begin{proof} Since $\Aut(M)$ is non-discrete, $A$ is finite, and $t\in S_1(A)$ is discerning, $t$ has infinitely many realizations. Since $B$ is finite, there is a realization of $t$ outside of $B$.
\end{proof}

\begin{lem}\label{type2} Let $p$ be any partial automorphism of $M$ and let $a\not\in \dom(p)$. Then there is a partial automorphism $q$ extending $p$ such that $q(a)\not\in \dom(q)$.
\end{lem}
\begin{proof} Since $M$ is ultrahomogeneous, $p$ can be extended to a $\psi\in\Aut(M)$. Let $t$ be the 1-type of $\psi(a)$ over $\rng(p)$. Since $a\not\in\dom(p)$, $\psi(a)\not\in\rng(p)$, and therefore $t$ is nontrivial. By Lemma~\ref{type1}, $t$ has a realization $b\not\in \dom(p)\cup\{a\}$. Since $b\in[t]$, the map $p\cup\{(a, b)\}$ is a partial automorphism. Let $q=p\cup\{(a, b)\}$, then $q$ extends $p$. Moreover, $q(a)=b\not\in \dom(p)\cup\{a\}=\dom(q)$.
\end{proof}

\begin{lem}\label{type3}Let $g\in\Aut(M)$, $p$ be a partial automorphism of $M$, and let $a\not\in \dom(p)$. Then there is a partial automorphism $q$ extending $p$ such that $q(a), gq(a)\not\in \dom(q)$ and $gq(a)\neq q(a)$.
\end{lem}

\begin{proof} Again, since $M$ is ultrahomogeneous, $p$ can be extended to a $\psi\in\Aut(M)$. Let $t$ be the 1-type of $\psi(a)$ over $\rng(p)$. Since $a\not\in\dom(p)$, $\psi(a)\not\in\rng(p)$, and therefore $t$ is nontrivial. By Lemma~\ref{type1}, $t$ has a realization $b\not\in \dom(p)\cup\{a\}\cup g^{-1}(\dom(p))\cup \{g^{-1}(a)\}$. Let $s$ be the 1-type of $b$ over $\dom(p)\cup\{a\}\cup g^{-1}(\dom(p))\cup\{g^{-1}(a)\}$. Then $s$ is nontrivial, hence is discerning. Let $c$ be a realization of $s$ with $g(c)\neq c$. Then $g(c)\not\in \dom(p)\cup\{a\}$, because otherwise $c\in g^{-1}(\dom(p)\cup\{a\})$ and would satisfy the negation of some formula in $s$. Now $t\subseteq s$, so we have that $c$ is a realization of $t$. As before, it follows that the map $p\cup\{(a, c)\}$ is a partial automorphism. Let $q=p\cup\{(a, c)\}$, then $q$ extends $p$. Moreover, $q(a)=c, gq(a)=g(c)\not\in \dom(p)\cup\{a\}=\dom(q)$.
\end{proof}

Let $p$ be any partial automorphism of $M$ and fix $a\not\in \dom(p)\cup \rng(p)$. By applying Lemma~\ref{type2} repeatedly, to either $p$ if $m_k>0$ or $p^{-1}$ if $m_k<0$, we obtain a partial automorphism $q$ extending $p$ or $p^{-1}$ accordingly, such that
$$ a, q(a), q^2(a), \dots, q^{|m_k|-1}(a) $$
are pairwise distinct. Then, by applying Lemma~\ref{type3} once to $g_k$, $q$ and $q^{|m_k|-1}(a)$, we obtain a partial automorphism $r$ extending $q$, such that
$$ a, r(a), r^2(a), \dots, r^{|m_k|-1}(a), r^{|m_k|}(a), g_kr^{|m_k|}(a) $$
are pairwise distinct. Let $p_k=r$ if $m_k>0$ and $p_k=r^{-1}$ if $m_k<0$. Then we get a partial automorphism $p_k$ extending $p$ such that $a, g_kp_k^{m_k}(a)$ are distinct.

Repeating the argument $k-1$ more times, we successively obtain partial automorphisms 
$$ p\subseteq p_k \subseteq p_{k-1} \subseteq \cdots \subseteq p_1 $$
such that
$$ a,\ g_kp_k^{m_k}(a),\ g_{k-1}p_{k-1}^{m_{k-1}}g_kp_{k-1}^{m_k}(a),\ \dots,\ g_1p_1^{k_1}\cdots g_kp_1^{m_k}(a) $$
are all pairwise distinct. In the very last step of the construction, we apply Lemma~\ref{type3} as above if $g_1\neq 1$ and apply Lemma~\ref{type2} if $g_1=1$. In particular, we obtain that for any $g\in\Aut(M)$ extending $p_1$, we have that
$g_1g^{m_1}\cdots g_kg^{m_k}(a)\neq a$. Since $p$ was arbitrary, we conclude that the set of $g$ such that $g_1g^{m_1}\cdots g_kg^{m_k}\neq 1$ is dense.
\end{proof}

The following corollary follows immediately from Proposition~\ref{UDeltadiscerning}, Proposition~\ref{subgroupMIF}, and Theorem~\ref{TheoremMIF}.

\begin{cor} For any countable distance value set $\Delta$ of cardinality $\geq 2$, $\Iso(\U_\Delta)$ is MIF. Moreover, any dense subgroup of $\Iso(\U_\Delta)$ is MIF.
\end{cor}





\begin{example}
\begin{enumerate}
\item $\mathbb{H}$ is MIF. 
\item Let $\mbox{Alt}(\mathbb{N})$ be the group of all finitely supported even 
permutations on $\mathbb{N}$. Hull--Osin showed that $\mbox{Alt}(\mathbb{N})$ 
is not MIF (Theorem 5.9 of \cite{HO}). Since $\mbox{Alt}(\mathbb{N})$ is dense in 
$S_\infty$, $S_\infty$ is not MIF.
\item Let $P$ be a set of primes and $T\subseteq \mathbb{H}$ be a subset of type $P$. The group $G_T$ constructed in Subsection~\ref{omnigenousfamily} is MIF. To see this, assume $\gamma_1x^{m_1}\cdots \gamma_kx^{m_k}$ is a mixed identity in $G_T$. By observing its action on $Y=\mathbb{H}$, there are $g_1\in\mathbb{H}$ and $g_2 \dots, g_k\in \mathbb{H}\setminus \{1\}$ such that for any $g\in \mathbb{H}$, $g_1g^{m_1}\cdots g_kg^{m_k}=1$. This contradicts the fact that $\mathbb{H}$ is MIF.
\item In Subsection~\ref{HallGroup} we noted that the family of groups $\mathbb{H}\oplus A$, where $A$ is an abelian $p$-group, consists of continuum many non-isomorphic countable universal locally finite groups. It is easy to see that $\mathbb{H}\oplus A$, when $A$ is nontrivial, is not MIF: a nontrivial mixed identity in $\mathbb{H}\oplus A$ is $w(x)=xgx^{-1}g^{-1}$, where $g\in A\setminus\{1\}$. Thus none of the groups $\mathbb{H}\oplus A$ are embeddable as a dense subgroup of $\Iso(\mathbb{QU})$ when $A$ is nontrivial.
\end{enumerate}
\end{example}

We mention another corollary of our results.

\begin{cor}
The group $\Iso(\U)$ is MIF. Consequently, any dense subgroup of $\Iso(\U)$ is MIF.
\end{cor}

\begin{proof}
Consider a nontrivial word $w(g;g_1,\ldots,g_k)$ with  $g_1,\ldots,g_k \in \Iso(\U)$. Using Lemma 5.1 of \cite{BM2}, we obtain a distance value set $\Delta$ and a dense subspace $X$ of $\U$ which is isometric to $\U_\Delta$ and such that $g_i(X)=X$ for all $i \in \{1,\ldots,k\}$. Since $\Iso(\U_\Delta)$ is MIF, we can find $g \in \Iso(X)$ such that $w(g;{g_1}|_X,\ldots,{g_k}|_X) \ne 1$. Extending $g$ to an isometry of $\U$, we have $w(g;g_1,\ldots,g_k) \ne 1$.
\end{proof}

In the rest of this subsection we consider some notions stronger than MIF. For notational simplicity we will write a word $w(x_1, \dots, x_n)\in G*F_n\setminus G$ as $w(x_1,\dots, x_n; g_1, \dots, g_k)$ if the constants occurring in the normal form of $w(x_1,\dots, x_n)$ are among $g_1, \dots, g_k$.

\begin{defn} Let $G$ be a group and $k\geq 1$ be an integer. We say that $G$ is \emph{$k$-MIF} if for any $w_1(x_1, \dots, x_n), \dots, w_k(x_1, \dots, x_n)\in G*F_n\setminus G$ there are $h_1, \dots, h_n\in G$ such that $w_1(h_1, \dots, h_n), \dots, w_k(h_1, \dots, h_n)\neq 1$. 

If $G$ is a locally finite group, we say that $G$ is \emph{$\infty$-MIF} if for any $g_1, \dots, g_k\in G$ and any infinite sequence $w_1(x_1, \dots, x_n; g_1, \dots, g_k), w_2(x_1,\dots, x_n; g_1, \dots, g_k), \dots$ of elements of $G*F_n\setminus G$, whenever there is a finite group $\Gamma$ which is an overgroup of $\langle g_1, \dots, g_k\rangle$ in which there are $\gamma_1,\dots, \gamma_n\in \Gamma$ such that $w_i(\gamma_1,\dots, \gamma_n; g_1, \dots, g_k)\neq 1$ for all $i\geq 1$, there are $h_1, \dots, h_n\in G$ such that $w_i(h_1,\dots, h_n; g_1, \dots, g_k)\neq 1$ for all $i\geq 1$.
\end{defn}


By the remark following Definition~\ref{def:MIF}, for each integer $k\geq 1$, the definition of $k$-MIF is equivalent to the version where we consider words with only one variable. 

\begin{prop}\label{prop: inftymif} Let $G$ be any group. Then the following 
hold:
\begin{enumerate}
\item[(i)] $G$ is MIF iff $G$ is $k$-MIF for some $k\geq 1$ iff $G$ is $k$-MIF 
for all $k\geq 1$.
\item[(ii)] If $G$ is a locally finite group and $G$ is $\infty$-MIF, then $G$ 
is MIF.
\item[(iii)] If $G$ is a locally finite group, then $G$ is $\infty$-MIF iff $G$ 
is omnigenous.
\end{enumerate}
\end{prop}

\begin{proof} 

MIF is exactly $1$-MIF, and it is clear that $k$-MIF implies $1$-MIF. The fact 
that $1$-MIF implies $k$-MIF is a direct consequence of Proposition 5.3 of 
\cite{HO}. We give a full proof here for the convenience of the reader. We 
first show that $1$-MIF implies $2$-MIF. Let $w_1(x_1, \dots, x_n), w_2(x_1, 
\dots, x_n)\in G*F_n\setminus G$. Let $w(x_1, \dots, x_n, x)\in G*F_n*\langle 
x\rangle$ be the word $[w_1, x^{-1}w_2x]=w_1^{-1}x^{-1}w_2^{-1}xw_1x^{-1}w_2x$. 
Then $w$ is nontrivial. Since $G$ is $1$-MIF, there are $h_1, \dots, h_n, h\in 
G$ such that $w(h_1,\dots, h_n, h)\neq 1$. It follows that $w_1(h_1, \dots, 
h_n), w_2(h_1, \dots, h_n)\neq 1$. Thus $G$ is $2$-MIF. In general, let $k\geq 
2$ and $w_1, \dots, w_k\in G*F_n\setminus G$. Then similarly define a 
nontrivial $w\in G*F_n*\langle x\rangle$ as 
$$ w=[w_1, x^{-1}w_2x, x^{-2}w_3x^2, \cdots, x^{-k+1}w_kx^{k-1}] $$
where the commutator is inductively defined by
$$ [u_1, u_2]=u_1^{-1}u_2^{-1}u_1u_2 \ \mbox{ and }\  [u_1, \dots, u_m]=[[u_1, 
\dots, u_{m-1}], u_m]. $$
Reasoning as before, we find $h_1,...,h_n, h\in G$ satisfying 
$w(h_1,...,h_n,h)\neq 1$, which implies that all the elements 
$w_1(h_1,...,h_n)$,..., $w_k(h_1,...,h_n)$ are non trivial. This finishes the 
proof that $1$-MIF implies $k$-MIF for every $k\geq 2$, so (i) is established.

To prove (ii), let $G$ be an $\infty$-MIF locally finite group. In view of (i), we only need to show that $G$ is MIF.
Let $w(x; g_1,...,g_k)$ be a 
nontrivial word in $G*\langle x\rangle$. Since $\langle g_1,...,g_k\rangle $ is finite and 
$\langle x\rangle \cong\mathbb{Z}$ is residually finite, the group $\langle g_1,...,g_k\rangle *\langle x\rangle$ is 
residually finite. We can thus find a finite group $\Gamma$ and a homomorphism 
$\rho :\langle g_1,...,g_k\rangle *\langle x\rangle\to\Gamma$ such that $\rho(w)\neq 1$ and 
$\rho(g)\neq 1$
 for every $g\in \langle g_1,...,g_k\rangle\setminus\{1\}$. Since $\rho\!\upharpoonright\!\langle g_1, \dots, g_n\rangle$ is an isomorphism, we 
 can then view $\Gamma$ as an overgroup of $\langle g_1,...,g_k\rangle$, in which $w(\rho(x); g_1, \dots, g_k)\neq 1$. Applying the 
 $\infty$-MIF property, we find $h\in G$ such that 
 $w(h;g_1,...,g_k)\neq 1$ as required.

To prove (iii), suppose $G$ is a locally finite group. First, assume $G$ is 
omnigenous. Let $g_1, \dots, g_k\in G$ and $w_1, w_2, \dots\in G*F_n\setminus 
G$ be given. Let $\Gamma$ be a finite overgroup of $\langle g_1, \dots, 
g_k\rangle$ and $\gamma_1, \dots, \gamma_n\in \Gamma$ are such that 
$w_i(\gamma_1,\dots, \gamma_n; g_1, \dots, g_k)\neq 1$ for all $i\geq 1$. Since 
$G$ is omnigenous, there are $h_1, \dots, h_n\in G$ such that the map given by 
$g_j\mapsto g_j$ for $j=1, \dots, k$ and $h_l\mapsto \gamma_l$ for $l=1,\dots, 
n$ generates a homomorphism from $\langle h_1,\dots, h_n, g_1, \dots, 
g_k\rangle$ onto $\langle \gamma_1, \dots, \gamma_n, g_1, \dots, g_k\rangle\leq 
\Gamma$. Thus, for any $i\geq 1$, $w_i(h_1, \dots, h_n; g_1, \dots, g_k)\neq 
1$, since otherwise by applying the homomorphism we would get $w_i(\gamma_1, 
\dots, \gamma_n; g_1, \dots, g_k)=1$.

Conversely, assume $G$ is $\infty$-MIF. Suppose $G_1=\{1\}\cup\{g_1, \dots, 
g_k\}$ is a finite subgroup of $G$ and $\Gamma$ is a finite overgroup of $G_1$ 
with 
additional generators $\gamma_1, \dots,\gamma_n\not\in G_1$, i.e., 
$\Gamma=\langle g_1, \dots, g_k,\gamma_1,\dots, \gamma_n\rangle$. Let $w_1, 
w_2, \dots$ enumerate all words $w\in G_1*F_n\setminus G_1$ such that 
$w(\gamma_1,\dots, \gamma_n; g_1, \dots, g_k)\neq 1$. Since $G$ is 
$\infty$-MIF, there are $h_1, \dots, h_n\in G$ such that for all $i\geq 1$, 
$w_i(h_1,\dots,h_n; g_1, \dots, g_k)\neq 1$. Then the map given by $g_j\mapsto 
g_j$ for $j=1,\dots, k$ and $h_l\mapsto \gamma_l$ for $l=1,\dots, n$ generates 
a homomorphism from $\langle g_1, \dots, g_n, h_1,\dots, h_n\rangle$ onto 
$\Gamma$. To see this, note that if $w(x_1, \dots, x_n; g_1, \dots, g_k)\in 
G_1*F_n\setminus G_1$ is such that $w(h_1,\dots, h_n; g_1,\dots, g_k)=1$, then 
$w(\gamma_1,\dots, \gamma_n; g_1, \dots, g_k)=1$ by our construction.
\end{proof}

We have thus established the following implications for countable locally finite groups $G$:
$$\begin{array}{c}\mbox{ $G$ is omnigenous} \\ \Updownarrow \\ \mbox{ $G$ is $\infty$-MIF} \\ \Downarrow \\ \mbox{ $G$ is embeddable as a dense subgroup of $\Iso(\U_\Delta)$ for $|\Delta|\geq 2$} \\ \Downarrow \\ \mbox{ $G$ is $k$-MIF for all $k\geq 1$} \\ \Updownarrow \\
\mbox{ $G$ is MIF } 
\end{array} $$
It follows from our results that $\mathbb{H}$ is $\infty$-MIF, along with all 
groups $H_P$ constructed in Subsection~\ref{omnigenousfamily}. We now show that 
being $\infty$-MIF is not equivalent to being MIF. To this end, we will use the 
following recent result of Jacobson, which relies on a construction due to 
Ol'shanskii, Osin and Sapir \cite{OOS}.

\begin{thm}[\cite{jacobsonMixedIdentityfreeElementary2019}]
	There is a MIF locally finite $p$-group.
\end{thm}
Jacobson's examples are explicit, but let us observe that 
the above result can be shown directly as follows.
\begin{proof}
	We construct 
	an increasing sequence of finite $p$-groups $\{G_i\}_{i=1}^\infty$ such 
	that $G=\bigcup_{i=1}^\infty G_i$ is MIF. 
	Let $G_1$ be any finite $p$-group and assume that we have constructed the 
	sequence $\{G_i\}_{i=1}^n$ of finite 
	$p$-groups. For every $1\leq i\leq n$, let $W_i=\{w_{i1},w_{i2},\dots\}$ be the 
	set of all nontrivial mixed identities 
	of $G_i$ of the form $w(g;g_1,\dots,g_j)$ where $g_1,g_2,\dots,g_j\in G_i$. 
	Furthermore, let 
	$\sigma:\mathbb{N} \times \mathbb{N} \rightarrow \mathbb{N}$ be a bijection 
	such that for every $s,t\in \mathbb{N}$ we have $\max\{s,t\}\leq 
	\sigma(s,t)$. We find $G_{n+1}$ such that:
	\begin{enumerate}
		\item $G_{n+1}$ is a finite $p$-group,
		\item $G_{n}\leq G_{n+1}$, and
		\item $w_{\sigma^{-1}(n)}$ is not a nontrivial mixed identity of 
		$G_{n+1}$.
	\end{enumerate}
	Assume $w_{\sigma^{-1}(n)}=g_1g^{n_1}\cdots g_kg^{n_k}$ where 
	$g_1,\dots,g_k\in G_n$ (this is possible since $\sigma^{-1}(n)=(s,t)$ 
	where $s,t\leq n$). If $w_{\sigma^{-1}(n)}$ is not a nontrivial mixed 
	identity of $G_n$, then we can take $G_{n+1}=G_n$. 
	Otherwise, $w_{\sigma^{-1}(n)}\in W_n$. Let $m\in \mathbb{N}$ be such that 
	$n_1,n_2,\dots,n_k< p^m$. Then $w_{\sigma^{-1}(n)}$
	is not a mixed identity of the free product 
	$H=G_n*\mathbb{Z}/p^m\mathbb{Z}$. Since both $G_n$ and 
	$\mathbb{Z}/p^m\mathbb{Z}$ are 
	$p$-groups, by Higman \cite{Hig} $H$ is a residually $p$-finite group. Let 
	$w'_{\sigma^{-1}(n)}=[w_{\sigma^{-1}(n)},h_1,\dots,h_l]$ where 
	$G_n=\{1,h_1,h_2, \dots,h_l\}$. Since $H$ is residually $p$-finite there 
	exists $N\trianglelefteq H$ such that $[H:N]$ is equal to a power of $p$ 
	and $w'_{\sigma^{-1}(n)}\notin N$. Now $G_{n+1}=H/N$ is as desired. By 
	construction, it is clear that $G=\bigcup_{i=1}^\infty G_i$ is MIF.
\end{proof}

\begin{prop}\label{pgroupMIF}
	There is a MIF locally finite group which is not $\infty$-MIF.
\end{prop}
\begin{proof}
	By Proposition \ref{prop: inftymif} it suffices to construct a MIF locally 
	finite 
	group which is not omnigenous. 
	Consider a MIF $p$-group $G$ as constructed above, let $G_1$ be any finite 
	subgroup of $G$. Then consider some $q\geq 2$ prime with $p$, and embed $G_1$ into $G_1\times 
	\Z/q\Z$. Since $G$ is a $p$-group, it is clear that no subgroup of $G$ 
	surjects onto $G_1\times \Z/q\Z$, so $G$ fails to be omnigenous as wanted.
\end{proof}


Note that a $p$-group can not be dense in $\Iso(\U_\Delta)$ since for every $q\in \mathbb{N}$ (in particular for $q\in \mathbb{N}$ where $(q,p)=1$) there exists $g\in \Iso(\U_\Delta)$ such that $g$ 
has an orbit of length $q$. Therefore, by Proposition \ref{pgroupMIF} there are MIF groups that are not dense in $\Iso(\U_\Delta)$.

Let us end this section with another class of potentially relevant examples in 
the above implications, namely some MIF highly transitive groups coming from 
Cantor dynamics. Recall that, if 
$\varphi$ is a minimal homeomorphism of a Cantor space $X$, its topological 
full group $[[\varphi]]$ is the group of all
homeomorphisms $g$ of $X$ such that 
there is a clopen partition $A_1,\ldots,A_n$ of $X$ and integers 
$i_1,\ldots,_n$ with the property that $g(x)=\varphi^{i_k}(x)$ for all $x \in 
A_k$. These groups are amenable (by a celebrated result of Juschenko-Monod 
\cite{JuMo}), and provide many examples of highly transitive countable amenable 
groups (their action on any $\varphi$-orbit is highly transitive and faithful). 
Now if one fixes some $x_0 \in X$ and consider the subgroup $\Gamma$ of all 
elements of $[[\varphi]]$ which map the positive semi-orbit of $x_0$ on itself, 
then $\Gamma$ is locally finite and also highly transitive. A basic example is 
the group of dyadic permutations, obtained for $\varphi$ being the odometer. 
All these groups are MIF; we would guess they cannot be embedded as dense 
subgroups of any $\Iso(\U_\Delta)$ besides $S_\infty$. If this 
were true, then in particular
the stabilizers of positive semi-orbits would never be $\infty$-MIF.

\section{A Characterization of Isomorphism\label{sec:cha}}

In this section we turn to a different problem, namely to study how the isomorphism type of $\Iso(\U_\Delta)$ is dependent on $\Delta$. 

\subsection{Topological simplicity of $\Iso(\U_\Delta)$}
The main theorem of this subsection is a result about topological simplicity for pointwise stabilizers in $\Iso(\U_\Delta)$, which we will use later. Tent--Ziegler \cite{TZ}, \cite{TZ2} have studied simplicity for many related automorphism groups, and the result here also follows from their techniques. We give a direct proof using techniques developed in the previous section. 


\begin{lem}\label{l:guirlande}
Let $\Delta$ be any countable distance set. Let $A\subseteq \U_\Delta$ be a finite set, $p$ a nontrivial $1$-type over $A$, $x, y\in [p]$, and $s \le 2 d(x,A)$ an element of $\Delta$. Then there exist an integer $m$ and elements $x_0=x, x_1\ldots, x_{m-1}, x_m=y \in [p]$ such that $d(x_{i-1},x_i)=s$ for all $i=1, \dots, m$.
\end{lem}

\begin{proof} Since $x$ and $y$ have the same $1$-type over $A$, we have $d(x, y)\leq 2d(x, A)$. If $s\geq d(x, A)$, then consider the Kat\u{e}tov map $f: A\cup\{x, y\}\to \Delta$ defined by
$$ f(a)=\left\{\begin{array}{ll} d(x, a), & \mbox{ if $a\in A$,}\\ s, & \mbox{ if $a\in\{ x, y\}$.}\end{array}\right. $$
Let $x_1\in \U_\Delta$ be such that $f(x_1)=d(x_1, a)$ for all $a\in A\cup\{x,y\}$. Then $x_1\in[p]$ and $d(x, x_1)=d(x_1, y)=s$. If $s<d(x, A)$, then let $m$ be the least integer so that $ms\geq 2d(x, A)$. Define a metric space $A\cup\{x=x_0, x_1, \dots, x_{m-1}, x_m=y\}$ as an extension of $A\cup\{x, y\}$ by letting, for all $1\leq i<m$,
$d(x_i, a)=d(x, a)$ for all $a\in A$ (i.e. $x_i\in[p]$), and for all $1\leq i\leq m$ and $j<i$, $d(x_j, x_i)=(i-j)s$ if $i-j<m$. By the Urysohn property of $\U_\Delta$, we can find such $x_1, \dots, x_{m-1}\in \U_\Delta$.
\end{proof}

For a finite subset $A\subseteq \U_\Delta$, denote by $G_A$ the pointwise stabilizer of $A$, i.e., $G_A=\{g\in\Iso(\U_\Delta): \forall a\in A\ g(a)=a\}$.

\begin{lem}\label{l:inductivestep}
Let $\Delta$ be any countable distance value set with $|\Delta|\geq 2$. Let $A\subseteq \U_\Delta$ be a finite set, $p$ a nontrivial $1$-type over $A$, and distinct $x, y\in[p]$. Let $g$ be a nontrivial element of $G_A$. Then there exist $h_1,\ldots,h_n \in G_A$ such that $h_ngh_n^{-1}\cdots h_1gh_1^{-1}(x)=y$.
\end{lem}

\begin{proof} 
Since $\U_\Delta$ is discerning by Proposition~\ref{UDeltadiscerning}, there is some $z \in [p]$ such that $d(g(z),z)=s>0$. Since $\U_\Delta$ is homogeneous, there is $h_0\in \Iso(\U_\Delta)$ with $h_0(z)=x$. Then $d(h_0gh_0^{-1}(x), x)=d(g(z), z)=s$. Note that $s \le 2 d(x,A)$ since $g$ fixes $A$ pointwise, so by Lemma \ref{l:guirlande} we may find $x=x_0,x_1, \ldots, x_{n-1}, x_n=y\in[p]$ such that $d(x_{i-1},x_i)=s$ for all $i=1, \dots, n$. By the ultrahomogeneity of $\U_\Delta$ we can find $g_1,\ldots,g_n \in G_A$ such that $g_i(x)=x_{i-1}$ and $g_i(h_0gh_0^{-1}(x))= x_i$. Letting $h_i=g_ih_0$, we thus have $h_igh_i^{-1}(x_{i-1})=x_i$ for all $i=1, \dots, n$ and we are done.
\end{proof}

\begin{thm}\label{t:topsimple}
Let $\Delta$ be any countable distance value set. Let $A \subset \U_\Delta$ be a finite set. Then the pointwise stabilizer $G_A$ is topologically simple.
\end{thm}

\begin{proof} When $|\Delta|=1$, $\Iso(\U_\Delta)$ is isomorphic to $S_\infty$. It is well-known that $S_\infty$ is topologically simple, and in $S_\infty$ pointwise stabilizers of finite tuples are isomorphic to $S_\infty$. In the rest of this proof, we assume $|\Delta|\geq 2$. 
Let $g$ be a nontrivial element of $G_A$. We show that the normal subgroup of $G_A$ generated by $g$ is dense in $G_A$. Let $N$ be the set of all  products of conjugates of $g$ by elements of $G_A$, plus the identity. It suffices to show that $N$ is dense in $G_A$. 

For this we show that for any finite $A\subseteq \U_\Delta$ and any partial isometry $\psi$ such that $A\subseteq \dom(\psi)$ and $\psi(a)=a$ for all $a\in A$, there is an $h\in N$ extending $\psi$. Without loss of generality assume $\psi$ is not the identity. We prove this by induction on $|\dom(\psi)\setminus A|$. Suppose $\dom(\psi)=A\cup \{x_1,\dots x_n\}$, and let $y_i=\psi(x_i)$ for $i=1,\dots, n$. For $n=1$,  this is exactly the content of Lemma \ref{l:inductivestep}. Assume that $n>1$ and the statement has been proved for $n-1$. Using our inductive hypothesis, we find $f\in N$ such that $f(x_1)=y_1, \dots, f(x_{n-1})=y_{n-1}$. Let $\varphi=f\psi^{-1}$. Then $\dom(\varphi)=A\cup\{y_1,\dots, y_{n-1}, y_n\}$, $\varphi(a)=a$ for all $a\in A$, $\varphi(y_i)=y_i$ for $i=1, \dots, n-1$ and $\varphi(y_n)=f(x_n)$. Applying Lemma~\ref{l:inductivestep} to $A'=A\cup\{y_1,\dots, y_{n-1}\}$ and $\varphi$, we get $f'\in N$ such that $f'(a)=a$ for all $a\in A$, $f'(y_i)=y_i$ for $i=1, \dots, n-1$, and $f'(y_n)=f(x_n)$. Let $h={f'}^{-1}f$. Then $h\in N$, $h(a)=a$ for all $a\in A$, $h(x_i)=y_i$ for all $i=1,\dots, n$.
\end{proof}

\subsection{Open subgroups and reconstruction}

\begin{defn}
Let $\Delta$ be a distance value set. A \emph{$\Delta$-triangle} is a triple $(d_1,d_2,d_3)$ of elements of $\Delta$ which satisfies the triangle inequalities, i.e. for all distinct $i, j, k\in\{1, 2, 3\}$, $|d_j-d_k|\leq d_i\leq d_j+d_k$.
\end{defn}

\begin{defn}
Two distance value sets $\Delta, \Lambda$ are \emph{equivalent} if there exists a bijection $\theta \colon \Delta \to \Lambda$ such that, for any $(d_1,d_2,d_3) \in \Delta^3$, we have
$$(d_1,d_2,d_3) \text{ is a $\Delta$-triangle} \Leftrightarrow (\theta(d_1),\theta(d_2),\theta(d_3)) \text{ is a $\Lambda$-triangle}. $$
\end{defn}

The following is the main theorem of this subsection, whose remainder is 
devoted to its proof.
\begin{thm}\label{t:reconstruction}
Let $\Delta,\Lambda$ be two countable distance value sets. The following are equivalent:
\begin{enumerate}
\item\label{eq1} $\Delta$ and $\Lambda$ are equivalent.
\item\label{eq2} $\Iso(\U_\Delta)$ and $\Iso(\U_\Lambda)$ are isomorphic (as abstract topological groups).
\item\label{eq3} There exists a (continuous) group homomorphism $\varphi \colon \Iso(\U_\Delta) \to \Iso(\U_\Lambda)$ with dense image.
\end{enumerate}
\end{thm}

The implication $\eqref{eq1} \Rightarrow \eqref{eq2}$ is easy: starting from $\U_\Delta$, and a map $\theta$ witnessing that $\Delta$ and $\Lambda$ are equivalent, define a new distance $d_\theta$ on $\U_\Delta$ by setting $d_\theta(x,y)=\theta(d(x,y))$. Then $(\U_\Delta,d_\theta)$ is isometric to $\U_\Lambda$, and its isometry group has not changed under this operation.

The implication $\eqref{eq2} \Rightarrow \eqref{eq3}$ follows from the automatic continuity property for $\Iso(\U_\Delta)$, which in turn follows from the fact that it has ample generics (cf. Corollary 4.1 of \cite{S}).  

Next we prove $\eqref{eq3} \Rightarrow \eqref{eq2}$. We will need to use stabilizers of $\Iso(\U_\Delta)$ for elements $x\in \U_\Delta$, which we will denote as $G^\Delta_x$ for notational simplicity. More generally, for any tuple $\bar{x}\in \U_\Delta^n$, we also denote by $G^\Delta_{\bar{x}}$ the pointwise stabilizer of $\Iso(\U_\Delta)$ for $\bar{x}$.

We will be using the following theorem which is essentially due to Slutsky (cf. Theorems 4.12, 4.16 and Corollary 4.17 of \cite{S12}).

\begin{thm}[Slutsky \cite{S12}]\label{Slutsky} Let $\Delta$ be any countable distance value set. If $A, B$ are finite subsets of $\U_\Delta$, then $\langle G_A, G_B\rangle$ is dense in $G_{A\cap B}$.
\end{thm}

\begin{lem}\label{finiteorbit}
Let $V$ be a proper open subgroup of $\Iso(\U_\Delta)$. Then there exists an element $x \in \U_\Delta$ with a finite $V$-orbit.
\end{lem}

\begin{proof}
Assume that $V$ is open and has only infinite orbits. We claim that $V$ is dense, so that $V=\Iso(\U_S)$, since any open subgroup is also closed. Let $\bar a \in \U_\Delta^n$ be such that 
$G^\Delta_{\bar a} \subseteq V$. By Neumann's lemma (cf. Corollary 4.2.2 of \cite{HodgesBook}), there exists $v \in V$ such that $g(\bar a)\cap \bar a = \emptyset$. By Theorem~\ref{Slutsky}, $\langle G^\Delta_{\bar a}, G^\Delta_{g (\bar a)} \rangle$ is a dense subgroup of $\Iso(\U_\Delta)$. This group is contained in $V$, and we are done.
\end{proof}


\begin{lem}
Assume that $\tau$ is a Hausdorff, separable group topology on $\Iso(\U_\Delta)$ which admits a proper open subgroup. Then $\tau$ coincides with the usual Polish group topology of $\Iso(\U_\Delta)$.
\end{lem}


\begin{proof} Consider the identity embedding from $\Iso(\U_\Delta)$ into $(\Iso(\U_\Delta), \tau)$. By automatic continuity, this embedding is continuous. Let $V$ be a proper $\tau$-open subgroup. Then $V$ as the pullback is open for the usual Polish group topology. By the previous lemma, there is an $a\in \U_\Delta$ with a finite $V$-orbit $A$. Using universality and ultrahomogeneity of $\U_\Delta$, we may find a $g \in \Iso(U_\Delta)$ such that 
$g(A) \cap A=\{a\}$ is a singleton. Then $V \cap gVg^{-1}$ is $\tau$-open, and contained in $G^\Delta_a$. Thus $G^\Delta_a$ is $\tau$-open. Since $\Iso(\U_\Delta)$ acts transitively on $\U_\Delta$, any point stabilizer is thus $\tau$-open. Since the point stabilizers form a nbhd subbasis for $1$ in the usual Polish group topology, we have that $\tau$ is finer than the usual topology, hence they are the same.
\end{proof}

\begin{prop}
Assume that $\varphi \colon \Iso(\U_\Delta) \to S_\infty$ is a nontrivial group homomorphism. Then $\varphi$ is a topological group isomorphism onto its image.
\end{prop}

\begin{proof}
We know that $\varphi$ is continuous by automatic continuity. By 
Theorem~\ref{t:topsimple}, $\Iso(\U_\Delta)$ is topologically simple, so 
$\varphi$ is also injective. Let $\tau$ be the topology on $\Iso(\U_\Delta)$ 
pulled back via $\varphi$. Then $\tau$ is Hausdorff, separable and admits a 
proper open subgroup by the nature of the topology of $S_\infty$. By the 
previous lemma, $\tau$ coincides with the usual Polish topology on 
$\Iso(\U_\Delta)$. This means that $\varphi$  is a topological group 
isomorphism onto its image.
\end{proof}

Now the implication $\eqref{eq3} \Rightarrow \eqref{eq2}$ follows immediately from the previous proposition. In fact, let $\varphi: \Iso(\U_\Delta)\to \Iso(\U_\Lambda)$ be a nontrivial group homomorphism with a dense image. Since $\Iso(\U_\Lambda)$ is a closed subgroup of $S_\infty$, $\varphi$ satisfies the assumption of the previous proposition. It follows that $\varphi$ is a topological group isomorphism onto its image, which is closed in $\Iso(\U_\Lambda)$. Thus $\varphi$ is onto.

\begin{lem}
Let $V$ be a proper open subgroup of $\Iso(\U_\Delta)$. The following are equivalent:
\begin{enumerate}
\item[(i)] There exists a (unique) $x \in \U_\Delta$ such that $V=G^\Delta_x$.
\item[(ii)] $V$ is topologically simple, and is maximal among proper closed subgroups of $\Iso(\U_\Delta)$.
\end{enumerate}
\end{lem}

\begin{proof} First suppose $V=G^\Delta_x$ for some $x\in \U_\Delta$. Then by Theorem~\ref{t:topsimple}, $V$ is topologically simple. Assume $H$ is a proper closed overgroup of $G^\Delta_x$. Let $g\in H\setminus G^{\Delta}_x$. Then $g(x)\neq x$ and $G^\Delta_{g(x)}=gG^\Delta_x g^{-1}\leq H$. By Theorem~\ref{Slutsky}, $\langle G^\Delta_x, G^\Delta_{g(x)}\rangle$, which is a subgroup of $H$, is dense in $\Iso(\U_\Delta)$. Since $H$ is closed, $H=\Iso(\U_\Delta)$.

For the converse, assume that $V$ is an open subgroup which is maximal among proper closed subgroups of $G$. By Lemma~\ref{finiteorbit}, there is $a\in \U_\Delta$ with a finite $V$-orbit, which we denote as $A$. Now $\{g\in\Iso(\U_\Delta): g(A)=A\}$ is a closed subgroup of $\Iso(\U_\Delta)$ containing $V$. By maximality we know that $V=\{g \in G_S \colon g(A)=A\}$. We claim that $A=\{a\}$, and so $V=G^\Delta_a$. Otherwise, $A\neq \{a\}$, then $G^\Delta_A\leq V$ is a proper normal open subgroup of $V$, thus $V$ is not topologically simple. 

For uniqueness, if $V=G^\Delta_x=G^\Delta_y$ for $x\neq y$, then by Theorem~\ref{Slutsky} $\langle G^\Delta_x, G^\Delta_y\rangle$ is dense in $\Iso(\U_\Delta)$, which, together with the openness (and therefore closedness) of $V$, would imply that $V$ is not proper.
\end{proof}

Note that in the characterization above, we may replace the condition ``$V$ is topologically simple'' by ``$V$ has a comeagre conjugacy class" or ``$V$ has a dense conjugacy class", since these two properties are true for stabilizers of singletons (which have ample generics) while the other maximal open subgroups admit a proper clopen normal subgroup (hence do not have a dense conjugacy class).

We can now prove the last remaining implication $\eqref{eq2}\Rightarrow \eqref{eq1}$ of Theorem \ref{t:reconstruction}. Assume that $\varphi \colon \Iso(\U_\Delta) \to \Iso(\U_\Lambda)$ is an isomorphism. By the previous lemma, we know that for any $x \in \U_\Delta$ there exists a unique $y \in \U_\Lambda$ such that $\varphi(G^\Delta_x)= G^\Lambda_y $, and vice versa. We write $y=\psi(x)$; since $\varphi$ is an isomorphism, $\psi$ is a bijection from $\U_\Delta$ to $\U_\Lambda$.

Note that, for any $(x,y), (x',y') \in \U_\Delta$ the groups $G^\Delta_{\{x,y\}}=G^\Delta_x \cap G^\Delta_y$ and $G^\Delta_{\{x',y'\}}=G^\Delta_{x'} \cap G^\Delta_{y'}$ are conjugate iff $d(x,y)=d(x',y')$. Since $\varphi$ is a group isomorphism which maps point stabilizers to point stabilizers, this implies that 
$$\forall x,y \in \U_\Delta \quad d(x,y)=d(x',y') \Leftrightarrow d(\psi(x),\psi(y))= d(\psi(x'),\psi(y')).$$

Given $d \in \Delta$, we can then define $\theta(d) \in \Lambda$ by finding $x,y \in \U_\Delta$ such that $d(x,y)=d$, and then setting $\theta(d)=d(\psi(x),\psi(y))$. Then $\theta$ is a bijection from $\Delta$ to $\Lambda$. 
It remains to prove that $\theta$ is an equivalence between $S$ and $T$. Pick $(d_1,d_2,d_3) \in \Delta^3$. Then $(d_1,d_2,d_3)$ is a $\Delta$-triangle iff there exists $x,y,z \in \U_\Delta$ such that $d(x,y)=d_1$, $d(y, z)=d_2$ and $d(z,x)=d_3$. Then $(\psi(x),\psi(y),\psi(z)) \in \U_\Lambda^3$, so that $(\theta(d_1),\theta(d_2),\theta(d_3))=(d(\psi(x),\psi(y)),d(\psi(y),\psi(z)),d(\psi(z),\psi(x))$ is a $\Lambda$-triangle. By symmetry, $(d_1,d_2,d_3)$ is a $\Delta$-triangle if $(\theta(d_1),\theta(d_2),\theta(d_3))$ is a $\Lambda$-triangle.

We have thus completed the proof of Theorem~\ref{t:reconstruction}.

\section{Open Problems\label{sec:open}}

The first general problem is to characterize all dense countable locally finite 
subgroups of $\Iso(\U_\Delta)$.

When $|\Delta|=1$, $\Iso(\U_\Delta)$ is just $S_\infty$, so we are asking which 
locally finite groups are highly transitive. Note that the Hull-Osin 
dichotomy  states that such groups are either MIF or contain a normal 
subgroup isomorphic to $\mbox{Alt}(\mathbb{N})$. 

For $|\Delta|\geq 2$, the likely answer to the characterization problem for 
locally finite groups is a condition strictly in between $\infty$-MIF and MIF, 
but we do not know that the following question has a negative answer.

\begin{ques} For $|\Delta|\geq 2$, are all dense locally finite subgroups of $\Iso(\U_\Delta)$ $\infty$-MIF?
\end{ques}

The second general problem is to explore the possibility of the other extreme, namely to characterize the isomorphism type of $\Iso(\U_\Delta)$ by the isomorphism types of their countable dense subgroups.

\begin{ques} If $S, T$ are countable distance value sets that are not equivalent, is there always a dense countable (locally finite) subgroup of one of $\Iso(\U_S)$ and $\Iso(\U_T)$ that cannot be embedded into the other as a dense subgroup?
\end{ques}

\end{document}